\documentclass[11pt]{article}
\usepackage{mathrsfs}
\usepackage{mathtools}
\usepackage{amssymb}
\usepackage{amsmath}
\usepackage{amsthm}
\usepackage{amscd}
\usepackage{commath}
\usepackage{enumerate}
\usepackage{stmaryrd}
\usepackage{enumitem}
\usepackage{hyperref}
\usepackage{thmtools}
\usepackage{listings}
\usepackage{tikz}
\usepackage{tikz-cd}
\usepackage{xcolor}

\usepackage{comment}
\usepackage{lipsum}
\usepackage{appendix}
\usetikzlibrary{matrix,arrows,decorations.pathmorphing}
\usepackage{bm}

\newtheorem{theorem}{Theorem}[section]
\newtheorem{lemma}[theorem]{Lemma}
\newtheorem{proposition}[theorem]{Proposition}

\newtheorem{corollary}[theorem]{Corollary}
\newtheorem*{theorem*}{Theorem}
\newtheorem*{proposition*}{Proposition}
\newtheorem*{corollary*}{Corollary}

\theoremstyle{definition}
\newtheorem{definition}[theorem]{Definition}

\theoremstyle{remark}
\newtheorem{remark}[theorem]{Remark}

\DeclareMathOperator{\Spf}{Spf}

\newcommand{\Q}{\mathbb{Q}}
\newcommand{\Z}{\mathbb{Z}}
\newcommand{\PP}{\mathbb{P}}
\renewcommand{\P}{\mathbb{P}}

\newcommand{\C}{\mathbb{C}}
\newcommand{\F}{\mathbb{F}}

\newcommand{\Fpbar}{\overline{\mathbb{F}}_p}

\newcommand{\cL}{\mathcal{L}}

\newcommand{\cO}{\mathcal{O}}

\newcommand{\cF}{\mathcal{F}}
\newcommand{\cH}{\mathcal{H}}

\newcommand{\A}{\mathbb{A}}

\newcommand{\pdiv}{\mathscr{G}}

\newcommand{\abvar}{\mathscr{A}}

\DeclareMathOperator{\Spec}{\mathrm{Spec}}
\DeclareMathOperator{\geom}{\mathrm{geom}}
\newcommand{\Sp}{\mathrm{Sp}}
\newcommand{\Spa}{\mathrm{Spa}}

\newcommand{\an}{\mathrm{an}}

\newcommand{\Pic}{\mathrm{Pic}}

\newcommand{\Gm}{\mathbb{G}_m}

\newcommand{\gal}{\mathrm{Gal}}

\newcommand{\Ag}{\mathcal{A}_g}
\newcommand{\Aguni}{\underline{\mathscr{A}}}
\newcommand{\AgK}{\mathcal{A}_{g,\mathsf{K}}}
\newcommand{\AgKbb}{\mathcal{A}_{g,\mathsf{K}}^\mathrm{BB}}
\newcommand{\AgKtor}{\mathcal{A}_{g,\mathsf{K}}^{\mathrm{tor}}}
\newcommand{\AgKpar}{\mathcal{A}_{g,{\mathsf{K}_\mathrm{par}}}}

\newcommand{\Ann}{\mathsf{A}}

\newcommand{\Dstar}{\mathsf{D}^{\times}}
\newcommand{\sD}{\mathsf{D}}
\newcommand{\sK}{\mathsf{K}}
\newcommand{\sKpar}{\mathsf{K}_{\mathrm{par}}}
\newcommand{\Dab}{\mathsf{D}^{\times} \times \mathsf{D}^b}
\newcommand{\Db}{(\mathsf{D}^\times)^a}

\newcommand{\RZ}{\mathfrak{R}\mathfrak{Z}}
\newcommand{\fA}{\mathfrak{A}}

\newcommand{\fX}{\mathfrak{X}}

\DeclareMathOperator{\Hom}{Hom}
\DeclareMathOperator{\End}{End}

\DeclareMathOperator{\GSp}{GSp}
\DeclareMathOperator{\Fil}{Fil}

\DeclareMathOperator{\rig}{rig}
\DeclareMathOperator{\Frac}{Frac}
\DeclareMathOperator{\Tate}{\mathsf{T}}
\DeclareMathOperator{\Vdr}{(\mathcal{V},\nabla)_{\textrm{dR}}}
\DeclareMathOperator{\der}{der}
\DeclareMathOperator{\ad}{ad}
\DeclareMathOperator{\per}{per}
\newcommand{\kbar}{\overline{k}}

\DeclareMathOperator{\Proj}{Proj}
\DeclareMathOperator{\red}{red}
\DeclareMathOperator{\Bl}{Bl}
\DeclareMathOperator{\good}{good}
\renewcommand{\O}{\mathcal{O}}

 \newenvironment{itemize*}
  {\begin{itemize}[topsep=-\parskip+\jot,itemsep=-\parskip-\jot]}
  {\end{itemize}}

\newenvironment{enumerate*}
  {\begin{enumerate}[label=(\alph*),topsep=-\parskip+\jot,itemsep=-\parskip-\jot]}
  {\end{enumerate}}

\newenvironment{enumerate**}
  {\begin{enumerate}[label=(\roman*),topsep=-\parskip+\jot,itemsep=-\parskip-\jot]}
  {\end{enumerate}}
  
\title{$p$-adic hyperbolicity for moduli spaces of abelian motives}
\author{Abhishek Oswal,  Ananth N. Shankar, Xinwen Zhu \smallskip\\\emph{\MakeLowercase{with an appendix by} Anand Patel}}
\date{\today}

\begin{document}
\setlist[description]{font=\normalfont\itshape\textbullet\space}
\maketitle
\begin{abstract}
We prove that Shimura varieties of abelian type satisfy a $p$-adic Borel-extension property over discretely valued fields. More precisely, let $\mathsf{D}$ denote the rigid-analytic closed unit disc and $\mathsf{D}^{\times} = \mathsf{D} \setminus \{0\}$, let $X$ be a smooth rigid-analytic variety, and let $S(G,\mathcal{H})_{\mathsf{K}}$ denote a Shimura variety of abelian type with torsion-free level structure. We prove every rigid-analytic map defined over a discretely valued $p$-adic field $\mathsf{D}^{\times} \times X \rightarrow S(G,\mathcal{H})_{\sK}^{\an}$ extends to an analytic map $\mathsf{D} \times X \rightarrow (S(G,\mathcal{H})_{\mathsf{K}}^{\textrm{BB}})^{\an}$, where $S(G,\mathcal{H})_{\mathsf{K}}^{\textrm{BB}}$ is the Baily-Borel compactification of $S(G,\mathcal{H})_{\mathsf{K}}$. We also deduce various applications to algebraicity of analytic maps, degenerations of families of abeloids, and to $p$-adic notions of hyperbolicity. Along the way, we also prove an extension result for Rapoport-Zink spaces. 

\end{abstract}

\section{Introduction}\label{sec:intro}

In 1972, Borel proved what is now known as the Borel extension property for Shimura varieties.
\begin{theorem*}[Borel, Theorem A of \cite{borel}]\label{boreloriginal}
    Let $S(G,\mathcal{H})_{\sK}$ be a Shimura variety with torsion-free level structure. Let $D$ be the complex open unit disc, and let $D^\times = D\setminus \{0\}$ be the punctured open unit disc. Let $a, b$ be non-negative integers. Then, every holomorphic map $(D^\times)^a \times D^b\rightarrow S(G,\mathcal{H})_{\sK}^{\textrm{hol}}$ extends to a holomorphic map $D^{a+b}\rightarrow (S(G,\mathcal{H})_{\sK}^\textrm{BB})^{\textrm{hol}}$, where $S(G,\mathcal{H})_{\sK}^\textrm{BB}$ denotes the Baily--Borel compactification of $S(G,\mathcal{H})_\sK$.
\end{theorem*}
By combining this extension property with resolution of singularities and GAGA, Borel proves the following important algebraicity theorem in the theory of Shimura varieties.
\begin{theorem*}[Borel, Theorem 3.10 of \cite{borel}]
    Let $S(G,\mathcal{H})_{\sK}$ be as above, and let $X$ be a complex algebraic variety. Then, every holomorphic map $X^{\textrm{hol}}\rightarrow S(G,\mathcal{H})_{\sK}^{\textrm{hol}}$ is automatically an algebraic map. 
\end{theorem*}

The main purpose of this paper is to prove analogous results in the $p$-adic setting. To that end, let $F$ denote a discretely valued $p$-adic field\footnote{The reader may assume that $F$ is a finite extension of $\Q_p$, or of $W(\Fpbar)[1/p]$.}. Let $\sD$ denote the rigid-analytic closed unit disc over $F$ and set $\Dstar = \sD\setminus \{0\}$. The main theorem of this paper is: 

\begin{theorem}\label{main}
 Let $S(G,\mathcal{H})_\sK$ denote a Shimura variety of abelian type with torsion-free level structure defined over a number field contained in $F$. Let $X$ be a smooth rigid analytic variety over $F$. Then every rigid-analytic map $f: \Dstar \times X \rightarrow S(G,\mathcal{H})_{\sK}^\an$ defined over $F$ extends to an analytic map $\sD \times X \rightarrow (S(G,\mathcal{H})^{\textrm{BB}}_{\sK})^\an$, where $S(G,\mathcal{H})^{\textrm{BB}}_{\sK}$ is the Baily-Borel compactification of $S(G,\mathcal{H})_{\sK}$.     
\end{theorem}

Along the way, we also prove an extension theorem for Rapoport-Zink spaces $\RZ^{\rig}$ (see \autoref{sec:RZstuff} for a review) -- a result interesting in its own right:

\begin{theorem}\label{introRZextension}
    Let $X$ be as in \autoref{main}. Then every rigid-analytic map $\Dstar \times X \rightarrow \RZ^{\rig}$ defined over $F$ extends analytically to a map $\sD \times X \rightarrow \RZ^{\rig}$. 
\end{theorem}
\begin{remark}
    Note that the extension property for maps out of $\Dstar\times X$ as in the statement of \autoref{main} and \autoref{introRZextension} immediately implies the extension property for maps out of $(\Dstar)^a\times\sD^b$ as in the statement of Borel's theorem.
    The inverse implication, however, is not automatic in non-archimedean situation. Namely, although
    every boundary point $z \in \sD \times X \setminus \Dstar\times X$ admits a neighbourhood $U$, such that $U\cap (\Dstar \times X)\subset U$ is isomorphic to $\Dstar\times \sD^b \subset \sD^{1+b}$, these neighbourhoods (together with $\Dstar\times X$) need not give a(n admissible) cover of $\sD \times X$.  Nevertheless, we shall see in \autoref{hyperlocalextensionforX} that the inverse implication still holds when the target is quasi-projective.
\end{remark}

Our main theorems have several corollaries. The first is the $p$-adic analogue of Borel's algebraicity theorem. 
\begin{corollary}\label{thm:borelalg}
    Let $V/F$ be an algebraic variety. Let $S(G,\mathcal{H})_{\sK}$ denote a Shimura variety of abelian type with torsion-free level structure. Then, every rigid-analytic map $V^{\an} \rightarrow S(G,\mathcal{H})_{\sK}^\an$ over $F$ is algebraic.
\end{corollary}
The analytification $\Ag^{\an}$, of the moduli space of polarized abelian varieties $\Ag$, solves the moduli problem of polarized abeloids (see \cite{conradabeloids} for this fact and the definition of abeloids). Therefore, \autoref{thm:borelalg} has the following consequence to the alegbraicity of polarized families of abeloids. 
\begin{corollary}\label{algebraicityofabeloids}
    Let $V/F$ be an algebraic variety. Then every family $B/V^{\an}$ of polarized abeloids is the analytification of a polarized abelian scheme $A/V$. 
\end{corollary}

For toroidal compactifications of Shimura varieties, the extension property for maps from $\Dstar\times X$ as in \autoref{main} cannot hold in general. But the valuative criterion of properness implies the following weaker statement still holds.

\begin{corollary}\label{toroidalextension}
Let $S(G,\mathcal{H})_{\sK}$ be as above and let $S(G,\mathcal{H})^{\textrm{tor}}_{\sK}$ be some toroidal compactification. Then, every analytic map $\Dstar \rightarrow S(G,\mathcal{H})^{\an}$ defined over $F$ extends to an analytic map $\sD \rightarrow (S(G,\mathcal{H})^{\textrm{tor}}_{\sK})^\an$. 
\end{corollary}

Consequently, we have the following corollary. 
\begin{corollary}\label{extensionforabeloids}
    Let $B/\Dstar$ denote a family of polarized abeloids defined over $F$. Suppose that $B[n]$ is constant on $\Dstar$ for some $n \geq 3$. Then, $B/\Dstar$ extends to a polarized family of semi-abeloids $\bar{B} / \sD$. Further, if $B_s$ has good reduction for some point $s\in \Dstar$, then $\bar{B}/\sD$ is actually a family of abeloids. 
\end{corollary}
\autoref{thm:borelalg} and \autoref{main} also have similar implications to the algebraicity of analytic families of K3 surfaces over algebraic bases, and to degenerations of families of K3 surfaces -- where at least one fiber has good reduction -- over $\Dstar$.
We will prove all these corollaries in Section \ref{sec:last}.

\subsection{Applications to \texorpdfstring{$p$}{p}-adic hyperbolicity}
Our main theorems also have applications to $p$-adic notions of hyperbolicity. In \cite{Brunebarbe}, Brunebarbe discusses different notions of hyperbolicity over the complex numbers and amongst other things, shows that Shimura varieties satisfy these notions. In the $p$-adic setting, there have been several attempts at formulating analogues of hyperbolicity  (see for example \cite{javanpeykar-brodyhyperbolic} and \cite{Jacksonhyperbolicity}). Consider the following definition of Brody hyperbolicity over a discretely valued field: 
\begin{definition}\label{discretebrody}
    Let $S$ denote an algebraic variety over the discretely valued $p$-adic field $F$. We say that $S$ is \emph{discretely Brody hyperbolic} if:
    \begin{enumerate}
        \item For every abelian variety $A$ defined over a finite extension $F'$ of $F$, every map $A\rightarrow S \times_F F'$ is constant.
        \item Every rigid-analytic map $\Gm^{\an} \rightarrow (S \times_F F')^\an$ defined over a finite extension $F'$ of $F$ is constant.
    \end{enumerate}
\end{definition}
\autoref{thm:borelalg} implies that Shimura varieties of abelian type are discretely Brody hyperbolic! Indeed, \autoref{thm:borelalg} implies that any such map to a Shimura variety of abelian type $\Gm^\an \rightarrow S(G,\cH)^{\an}_{\sK}$ must be algebraic. There are no non-constant algebraic maps from $\Gm$ to a Shimura variety with torsion-free level structure (such a map would give a non-constant holomorphic map $\C \rightarrow S(G,\cH)_{\sK}^{\textrm{hol}}$, but it is known that there does not exist such map). In fact, \autoref{thm:borelalg} implies that every rigid-analytic map $U^{\an}\rightarrow S(G,\cH))_{\sK}$ defined over a finite extension of $F$ must also be constant, where $U$ is a smooth simply connected algebraic variety! Indeed, Deligne proves (in \cite{Delignefixedpart}) that every algebraic map $U\rightarrow S(G,\cH)_{\sK}$ from a smooth simply connected variety to a Shimura variety must be constant. The result in \emph{loc. cit.} proves that any map from a smooth variety with finite fundamental group must be constant. This can be leveraged to obtain the following corollary: 
\begin{corollary}\label{nomapsfromgroups}
    Let $\mathbb{G}$ denote an algebraic group defined over the discretely valued non-archimedean field $F$. Then, every rigid-analytic map $\mathbb{G}^{\an} \rightarrow S(G,\cH)^\an_{\sK}$ defined over a finite extension of $F$ is constant.
\end{corollary}
Deducing this corollary from \autoref{thm:borelalg} should be well known to experts. Nonetheless, we include a brief sketch in Section \ref{sec:last} for completeness. 

Using work of Nadel \cite[Theorem 0.1]{Nadel}, our algebraization result \autoref{thm:borelalg} implies the following corollary: 

\begin{corollary}\label{cor:Nadel}
    Let $N \geq \max \{28,g(g+1)/2 \}$ and let $\mathcal{A}_{g,1,N}$ denote the moduli space of principally polarized abelian varieties with full level $N$ structure. Then, there is no non-constant rigid-analytic morphism $C^{\an} \rightarrow \mathcal{A}_{g,1,N}^{\an}$ defined over $F$, where $C/F$ is a smooth algebraic curve whose (smooth) compactification has genus $\leq 1$. 
\end{corollary}

\subsection{Outline of proof}
Our proof does not parallel Borel's proof in the archimedean case. The proof of \autoref{main} blends $p$-adic and $\ell$-adic methods together with techniques from rigid-analytic geometry. To ease exposition, we will stick to the case of $\Dstar \rightarrow \Ag$, where $\Ag$ is the moduli space of polarized abelian varieties with suitably fine level structure. We will remark on the setting of $\Dstar \times X$ when appropriate. We write $\Dstar = \cup_{n\geq 0} \Ann_n$, where $\Ann_n$ is the closed annulus with outer radius 1 and inner radius $\frac{1}{p^n}$. The standing assumption will be that all spaces and maps are defined over the discretely valued field $F$. The outline goes as follows:

\subsection*{Reducing to the good reduction locus}
Working backwards, one should expect that any map $f: \Dstar \rightarrow \Ag^\an$ should have the property that $f(\Dstar_{\epsilon})$ is entirely contained inside the good-reduction locus, or entirely contained inside the bad reduction locus. Here, $\sD_{\epsilon} \subset \sD$ is the closed unit disk of radius $\epsilon$, and $\Dstar_{\epsilon} = \sD_{\epsilon} \cap \Dstar$. Using the $\ell$-adic monodromy of abelian varieties in conjunction with the interplay between formal schemes and rigid-analytic varieties, we in fact are able to prove that every $x\in f(\Dstar)$ has the same reduction type. It is striking that we do not have to shrink $\Dstar$! By analyzing the geometric structure of $\Ag^{\textrm{BB}}$ and $\Ag^{\textrm{tor}}$ around the boundary, we reduce to prove \autoref{main} in the case when $f(\Dstar)$ is entirely contained in the good reduction locus $\Ag^{\good}$ of $\Ag^{\an}$.

\subsection*{Lifting to Rapoport-Zink space}
In the appendix, Patel proves that the special fiber of a formal model $\fA$ of a closed rigid-analytic annulus must have a very specific form. Using this classification, we prove that the mod $p$ special fiber of a formal abelian scheme defined over $\fA$ is quasi-isogenous to some abelian variety defined over a finite field. Together with the fact that the map from a Rapoport-Zink space to $\Ag$ is a topological cover onto its image, we show that every map $f: \Dstar\rightarrow \Ag^\an$ whose image is contained in the good reduction locus lifts to a map $\Dstar\rightarrow \RZ$! This reduces \autoref{main} to \autoref{introRZextension}, i.e. to proving that Rapoport-Zink spaces satisfy the Borel-extension property. 

\subsection*{Extension for Rapoport-Zink spaces}
We thus reduce \autoref{main} to \autoref{introRZextension}.
As above, our exposition will be in the setting of a map $\tilde{f}: \Dstar \rightarrow \RZ$. There are two steps in proving that $\tilde{f}$ extends to an analytic map from $\sD$, and both steps make crucial use of the family of (fiberwise crystalline) Galois representations paremeterized by Rapoport-Zink spaces. First, we use inputs related to the $p$-adic Riemann-Hilbert correspondence \cite{DLLZ} and Shimizu's work \cite{koji} on $p$-adic local systems to prove that the family of (fiberwise crystalline) Galois representations on $\Dstar$ extends to a family of crystalline Galois representations on $\sD$. The second step involves a careful analysis of the period map $\check{\pi}: \RZ \rightarrow \cF$ where $\cF$ is an appropriate flag variety. We use Hartl's description of the image of the period map to prove that the map $\check{\pi} \circ \tilde{f}$ extends to a rigid-analytic map $g: \sD \rightarrow \cF$ with $\textrm{Im}(g) \subset \textrm{Im}(\check{\pi})$. Finally, we use De Jong's analysis of $\check{\pi}$ in \cite{deJong-etale} to extend $\tilde{f}$.

\subsection{Prior work}
In \cite{Cherry}, Cherry proves that for a smooth, projective algebraic curve over $\C_p$ of genus at least 2, any rigid-analytic map $\Dstar \rightarrow C^{\an}$ must extend analytically to a map $\sD \rightarrow C^{\an}$. Cherry and Ru in \cite{cherry-ru} prove a non-archimedean big Picard style result, using which Sun, in \cite{sun-hyperbolicity-of-mg}, proves the $\C_p$-analogue of \autoref{thm:borelalg} for $\mathcal{M}_g$, the moduli space of genus $g$ curves. We remark that our methods are completely different from these previous results, and in many cases give alternative proofs (including higher-dimensional extension theorems) for maps defined over discretely valued fields.

\subsection{Organization of the paper}
In Section \ref{sec:preliminaries}, we recall basic facts of Shimura varieties and Rapoport-Zink spaces, and reduce the main theorem to the case of $\Ag$. In Section \ref{sec:badred}, we reduce to the case of the good-reduction locus. In Section \ref{sec:lifttoRZ}, we prove that any map $\Dab\rightarrow \Ag^{\good}$ can be lifted to a map $\Dab \rightarrow \RZ$ for an appropriate Rapoport-Zink space. In Section \ref{sec:extension for RZ}, we prove \autoref{introRZextension}. Finally, in Section \ref{sec:last}, we put the results of previous sections together to deduce \autoref{main} and all the corollaries, including \autoref{thm:borelalg}.

\subsection{Notation and conventions}
Throughout we let $p$ denote a prime number, $F$ denote a discretely valued $p$-adic field. The reader may assume that $F$ is a finite extension of either $\Q_p$ or $W(\Fpbar)[1/p]$. We let $\cO_F$ denote the valuation subring of $F$ with uniformizer $\pi$, and $\F$ denote the residue field of $F$. We regard rigid-analytic varieties as adic spaces. We will freely use the equivalence of categories between quasi-separated  rigid-analytic varieties over $F$ in the sense of Tate and quasi-separated adic spaces locally of finite type over $F$. Given a rigid-analytic variety $\Sp(R)$ (in the sense of Tate) associated to an affinoid $F$-algebra $R$, we let $\Spa (R,R^+)$ denote the corresponding adic space associated to the affinoid algebra $(R,R^+)$ with $R^+ = R^\circ$ being the subring of $R$ consisting of power-bounded elements. Unless mentioned otherwise, any rigid-analytic variety we consider will be of locally finite type and any affinoid variety will be of topologically finite type. We will also use the fact that the category of taut adic spaces is equivalent to the category of Hausdorff strictly analytic Berkovich spaces.
For an algebraic variety $X$ over $F$, we denote by $X^\mathrm{an}$ the rigid analytic variety over $F$ associated to $X$. Given an admissible formal $\cO_F$-scheme $\mathfrak{X} \rightarrow \mathrm{Spf}(\cO_F)$ (in the sense of \cite{bosch}), we denote the rigid analytic generic fiber of $\mathfrak{X}$ (in the sense of Raynaud) by $\mathfrak{X}^\mathrm{rig}.$ Given a scheme $X \rightarrow \Spec \cO_F$ that is locally of finite type, we will let $X^{\rig}$ denote the rigid analytic generic fiber of the formal scheme obtained by $p$-adically completing $X$. The Tate algebra over $F$ in the variables $t_1,\ldots,t_n$ will be denoted by $F\langle t_1,\ldots,t_n \rangle$.

We denote by $\mathsf{D}$ the rigid analytic closed unit disk over $F$, and by $\Dstar := \mathsf{D}\setminus \{0\}$ the rigid closed unit disk over $F$ punctured at the origin. We let $\sD_{\epsilon}$ and $\Dstar_{\epsilon}$ denote the analogous objects with radius $\epsilon$.

For Shimura datum $(G,\mathcal{H})$ and a compact open subgroup $\mathsf{K} \subset G(\mathbb{A}_f)$, we say that the level structure $\mathsf{K}$ is \emph{torsion-free} if every connected component of the complex Shimura variety $S(G,\mathcal{H})_{\sK} = G(\Q)\backslash (\mathcal{H}\times G(\mathbb{A}_f))/\mathsf{K}$ is of the form $\Gamma_i\backslash \mathcal{H}_i$ for a connected component $\mathcal{H}_i$ of $\mathcal{H}$ and a \emph{torsion-free} congruence subgroup $\Gamma_i \subseteq G(\Q).$ We let $S(G,\mathcal{H})^{\textrm{BB}}_{\sK}$ denote the Baily-Borel compactification of $S(G,\mathcal{H})_{\sK}$.

\subsection*{Acknowledgements}
It is a pleasure to thank H\'el\`ene Esnault, Ben Howard, Kai-Wen Lan, Keerthi Madapusi, Frans Oort, Congling Qui, Michael Rapoport, Koji Shimizu and Salim Tayou for helpful comments and conversations. A.S. is partially supported by the NSF grants DMS-2337467 and DMS-2338942, and would like to thank MSRI for their hospitality during the spring of 2023.
X.Z. is partially supported by the National Science Foundation under agreement Nos. DMS-1902239 and DMS-2200940 and by a grant from the Institute for Advanced Study School of Mathematics.

\section{Preliminaries and reduction to \texorpdfstring{$\mathcal{A}_g$}{Ag}}\label{sec:preliminaries}
\subsection{Moduli of abelian varieties and $p$-divisible groups} \label{sec:RZstuff}

Let $(V,\psi)$ denote the unique non-degenerate symplectic $\Q$-vector space with dimension $2g$. Let $\GSp$ denote the group of symplectic similitudes over $\Q$ defined by $(V,\psi)$. Fix a lattice $L \subset V\otimes \Q_p$, which satisfies $L \subset L^{\vee} \subset \frac{1}{p}L$, where $L^{\vee} \subset V\otimes \Q_p$ denotes the lattice dual to $L$ under $\psi$. The subset of $\GSp(\Q_p)$ stabilizing the lattice chain $L\subset L^{\vee}$ is a maximal parahoric subgroup of $\GSp(\Q_p)$, and every maximal parahoric arises this way (see \cite[Chapter 3]{RZ}). We note that every compact open subgroup of $\GSp(\Q_p)$ is contained in a maximal parahoric subgroup. Recall that $\GSp$ admits a natural Shimura data. For $\sK \subset \GSp(\A_f)$ a compact open subgroup, we let $\AgK$ denote the associated Shimura variety viewed over $\Q$ with level structure defined by $\sK$. Recall that $\AgK$ can be identified with the moduli space of $g$-dimensional polarized abelian varieties. Whether or not the moduli problem is fine and indeed the precise moduli problem itself depends on the choice of level structure. We will now define various notions associated to level structure. 

Let $\sK_p = \sK \cap \GSp(\Q_p)$, and let $(\sKpar)_p$ denote a maximal parahoric subgroup containing $\sK_p$. 
\begin{definition}\label{defn: level at p and ell}
\begin{enumerate}
    \item We say that the level structure is \emph{maximal parahoric at $p$} if $\sK = \sK^p \sK_p$ with $\sK_p = (\sKpar)_p$ for some maximal parahoric subgroup, where $\sK^p$ is the prime-to-$p$ part of $\sK$. We denote the Shimura variety by $\AgKpar$ if the level structure is maximal parahoric at $p$.     

    \item Let $\ell$ be a prime. We say $\AgK$ has full level-$\ell$ structure if $\sK_{\ell}$ is contained in the principal congruence subgroup at $\ell$.
\end{enumerate}
\end{definition}

Suppose now that $\ell \neq p$ is an odd prime, and that $\AgKpar$ has full level structure at $\ell$. Then, the level structure is torsion free and $\AgKpar$ is a fine moduli space. We let $\Aguni_{\sKpar} \rightarrow \AgKpar$ denote the universal polarized abelian scheme over $\AgKpar$. For brevity we will refer to the universal family over $\AgKpar$ as $\Aguni \rightarrow \AgKpar$ for all choices of level structure as above. The family $\Aguni \rightarrow\AgKpar$ admits a canonical $\Z_p$-model, which we also denote by $\Aguni \rightarrow \AgKpar$. The abelian scheme $\Aguni\rightarrow \AgKpar$ is a polarized abelian scheme with prime-to-$p$ level structure, such that the $p$-power part of the kernel of the polarization is a $p$-torsion subgroup.

\begin{definition} \label{definition:TateVdR} We let $\Tate/\Ag\otimes \Q_p$ denote the relative $p$-adic \'etale cohomology and $\Vdr/\Ag$ denote the relative de Rham cohomology with Gauss-Manin connection of the universal abelian scheme over $\Ag\otimes \Q_p$.  
\end{definition}

We recall the definition of Rapoport-Zink spaces, which are local analogues of moduli of abelian varieties.
Let $\mathrm{Nilp}_{W(\Fpbar)}$ denote the category of $W(\Fpbar)$-algebras in which $p$ is nilpotent. We fix $L\subset L^\vee\subset \frac{1}{p}L \subset V\otimes \Q_p$ as before.

We fix a pair $\mathbb X=(\pdiv_0,\lambda_0)$ where $\pdiv_0$ is a $p$-divisible group over $\Fpbar$ and $\lambda_0: \pdiv_0\to \pdiv_0^\vee$ is a polarization with $\ker(\lambda_0)\subseteq \pdiv_0[p]$ and the degree of $\deg \lambda_0$ equals the index of $L$ in $L^{\vee}$.
\begin{definition}
We let
$\RZ_{\mathbb X}$ be the functor from $\mathrm{Nilp}_{W(\Fpbar)}$ to sets, such that $R\mapsto \{(\pdiv,\lambda,\rho) \}/\sim$ where $(\pdiv,\lambda)$ is a polarized $p$-divisible group over $S=\Spec R$, satisfying the following conditions: 
\begin{enumerate}
    \item $\rho: \pdiv_0 \times S_{\Fpbar} \rightarrow \pdiv \times_{S} S_{\Fpbar}$ is a quasi-isogeny of $p$-divisible groups.
    \item $\ker(\lambda) \subseteq \pdiv_0[p]$.
    \item The degree of the polarization $\deg \lambda$ equals the index of $L$ in $L^{\vee}$.
    \item The polarizations are compatible under pullback up to a scalar, i.e. $\rho^*(\lambda) = c\lambda_x$ for some  $c\in \Q_p^{\times}$
\end{enumerate}
Two data $(\pdiv,\lambda,\rho )$ and $(\pdiv',\lambda',\rho')$ defined over $R$ are said to be equivalent if there is an isomorphism between $(\pdiv,\lambda)$ and $(\pdiv',\lambda')$ whose mod $p$ reduction commutes with the quasi-isogenies from $\pdiv_0$.
\end{definition}
Note that Rapoport-Zink defined these spaces in a more general setting. We only consider special cases, namely those associated to $\GSp_{\Q_p}$ with maximal parahoric levels, which are sufficient for our purposes.

 Rapoport-Zink proved (\cite{RZ}) that $\RZ_{\mathbb X}$ is represented by a formal scheme (which we shall also denote by $\RZ_{\mathbb X}$) over $W(\Fpbar)$. Let $\RZ_{\mathbb X}^{\mathrm{rig}}$ denote its rigid generic fiber over $F=W(\Fpbar)[1/p]$. 
Let $\Tate/\RZ_{\mathbb X}^{\mathrm{rig}}$ denote the dual of the $p$-adic Tate module of the universal $p$-divisible group over $\RZ_{\mathbb X}^{\mathrm{rig}}$, which is a $\mathbb Z_p$-local system on $\RZ_{\mathbb X}^{\mathrm{rig}}$. One can also define the relative de Rham cohomology $\Vdr$ with the Gauss-Manin connection. However, by the universal quasi-isogney $\Vdr$ is canonically trivial, isomorphic to $(\mathbb{D}_0\otimes \mathcal O_{\RZ_{\mathbb X}^{\mathrm{rig}}},d)$, where $\mathbb{D}_0$ is the rational contravariant Dieudonne module of $\pdiv_0$. We record this for future use. 

\begin{lemma}\label{thm:triviality-of-vbconnection-RZ}
The universal vector bundle with connection restricted to the generic fiber $\Vdr|_{\RZ^{\rig}_{\mathbb{X}}}$ admits a canonical trivialization induced by the universal quasi-isogeny on $\RZ_{\mathbb{X}}$.
\end{lemma} 

The vector bundle with connection $\Vdr/\RZ_{\mathbb{X}}$ admits a one-step filtration $\Fil^1 \subset \mathcal{V}$. The canonical trivialization of $\Vdr|_{\RZ^{\rig}_{\mathbb{X}}}$ gives a canonical \emph{period map} 
\[
\per: \RZ^{\rig}_{\mathbb{X}} \rightarrow \cF
\]
to an appropriate flag variety, given by the variation of $\Fil^1$ with respect to the above trivialization of $\Vdr|_{\RZ^{\rig}_{\mathbb{X}}}$.

Now for a point $x\in \AgKpar(\overline{\F}_p)$, we let $\abvar_{x}$ and $\pdiv_{x}=\abvar[p^\infty]$ denote the polarized abelian variety and $p$-divisible group associated to $x$. By abuse of notation, we let $x$ denote the pair $(\pdiv_{x},\lambda_x)$.  
Rapoport-Zink also proved (\cite{RZ}) that $\RZ_x^{\mathrm{rig}}$ ``uniformizes'' a certain locus of $\AgKpar$. Indeed, there is a natural map 
\begin{equation}
    j: \RZ_x \rightarrow \widehat{\AgKpar}
\end{equation} 
 where $\widehat{\AgKpar}$ is the the formal scheme obtained by $p$-adically completing $\AgKpar$. The map $j$ can be described as follows. Let $\fX$ denote some locally noetherian $p$-adic formal $W(\Fpbar)$-scheme, and let $(\pdiv,\lambda,\rho) /\fX$ denote the data associated to some $\fX$-valued point of $\RZ_x$. Without loss of generality, by replacing $\rho$ by $p^n \rho$ for an appropriate integer $n$, we may assume that $\rho$ is an isogeny. Define $A'$ to be the abelian scheme defined to be $A_x \times \fX_{\Fpbar} / \ker \rho$. There is then a canonical isomorphism $i: A'[p^{\infty}] \rightarrow \pdiv \times \fX_{\Fpbar}$ that commutes with $\rho$ and the quotient map. The abelian scheme $A'$ is also equipped with a unique polarization $\lambda'$ that equals $\lambda$ on $\pdiv \bmod p$, and pulls back to $\lambda_{x}$ (upto a scalar) on $A_x$.  
The isomorphism $i$ allows us to view $\pdiv,\lambda$ as a deformation of $A'[p^{\infty}],\lambda'$. The Serre-Tate theorem yields a polarized abelian scheme $\tilde{\abvar'}/\fX$ (whose $p$-divisible group is canonically isomorphic to $\pdiv$) such that $\tilde{\abvar'}\times_{\fX} \fX_{\Fpbar}$ is canonically isomorphic to $A'$. The data of the (prime-to-$p$) level structure on $\abvar_{x}$ canonically defines a prime-to-$p$ level structure on $A'$, and therefore on $\tilde{\abvar'}$, whence we obtain an $\fX$-valued point of $\AgKpar$, as required.

The map $j|_{\RZ^{\rig}_x}$ (which we will also denote by $j$) is a ``topological covering" of its image in the following strong sense: First, the image $\mathfrak{T}$ consists of the tube inside $\AgKpar^\an$ over the locus in $\AgKpar \otimes \overline{\mathbb{F}}_p$ of points $p$-power isogenous to $\abvar_x$ (see \cite{RZ}, or \cite[Chapter II, Theorem 7.2.4]{Andre}), and therefore is a rigid analytic space itself. Secondly, the map $j: \RZ^{\rig}_x \rightarrow \mathfrak{T}$ has the property that, for any point $t \in \mathfrak{T}$, there exists an affinoid subdomain $U \subset \mathfrak{T}$ containing $t$ such that $j^{-1}(U)$ is the disjoint union of affinoid subdomains in $\RZ_x$, where on each of these affinoid subdomains the map $j$ restricts to an isomorphism onto $U$.

\subsection{Reduction to the case of \texorpdfstring{$\Dab \rightarrow \Ag$}{DabtoAg}}
In this section, we reduce the main theorem to the case of $\AgK$ with maximal parahoric level structure at $p$ and full level $\ell$ structure for some odd prime $\ell \neq p$. We will also prove that it suffices to show that every rigid-analytic map $f: \Dab \rightarrow \AgKpar^{\an}$ extends analytically to a map $f: \sD_{\epsilon}^{1+b} \rightarrow (\AgKpar^{\textrm{BB}})^{\an}$ for some $\epsilon > 0$, with $\sKpar$ as above. We will first need several preliminary results. 

\begin{theorem}\label{hyperlocalextensionforX}
Let $X$ denote a smooth rigid analytic variety over the non-archimedean field $F$. Let $f: \Dstar \times X \rightarrow (\PP^n)^{\an}$ denote a rigid-analytic map over $F$. Suppose that there exists a rational open neighbourhood $U_z \subset \sD \times X$ and an analytic extension of $f$ to $U_z$ for every classical point $z\in (\sD \times X) \setminus (\Dstar \times X).$

Then, $f$ extends analytically to a map $ \sD \times X \rightarrow (\PP^n)^\an$. 
\end{theorem}
\begin{proof}
First, by considering an open cover of $X$ by connected rational affinoid subdomains, we may assume that $X$ is smooth, connected and affinoid which we shall henceforth assume.

Let $x\in X$ be some classical point, and let $z = (0,x)\in  \sD \times X$. We let $f_z: (\Dstar \times X) \cup U_z \rightarrow (\PP^n)^\an$ denote the extension of $f$. We consider $f_z^{-1}(H) \subset (\Dstar \times X) \cup U_z$, where $H\subset \PP^n$ is a hyperplane. By \cite[Theorem 3.15 (b)]{lutkebohmert}, this extends to a closed analytic divisor $\Delta \subset \sD \times X$\footnote{Luktebohmert works in Tate's category of rigid-analytic varieties, but that is equivalent to our category of adic spaces.}. We let $\cL \in \Pic(\sD \times X)$ denote the line bundle associated to this divisor. Let $(a_0,...,a_n)$ denote a basis of sections of $\cO(1)$, and let $s_j = f_z^{*}(a_j)$. Note that $s_j$ is a section in $H^0(U_z \cup ( \Dstar\times X),\cL)$. We will now show that the $s_j$ extend to sections of $\cL$ on all of $ \sD \times X$, and that we may use them (together with hypothesis of local extensions) to extend $f$ analytically to a map $ \sD \times X \rightarrow \PP^n$. There exists a finite affinoid cover $X = \cup_{1 \leq i \leq m} X_i$ of $X$ with $U_z \subset X_1$ such that $X_i\cap X_{i+1}\neq \emptyset$ and such that $\cL|_{\sD_{\epsilon} \times X_i}$ is trivializable. Without loss of generality, we assume that $\epsilon =1$. 

\subsubsection*{Step 1: Extending the sections.}
It suffices to prove that the sections $s_j$ extend to sections of $\cL$ on $ \sD \times X_i$ for every $i$. We proceed to prove this inductively on $i$, starting with $i = 1$.
By construction, the sections $s_j$ already extend to $U_z$. We pick an isomorphism $\cL|_{ \sD \times X_1} \xrightarrow{\sim} \cO$. Under this isomorphism, the $s_j$ correspond to functions $g_j$ on $( \Dstar \times X_1) \cup U_z$.  In order to prove that the $s_j$ extend to $X_1\times \sD$, it suffices to prove that the functions $g_j$ do. But that just follows from \autoref{extfunction} below. 

We now suppose that we have extended the $s_j$ to $ \sD \times X_i$, and we will use this to extend them to $\sD \times X_{i+1}$. We have that the $s_j$ extend to $ (\sD \times (X_i \cap X_{i+1})) \cup ( \Dstar \times X_{i+1}) \subset ( \sD \times X_{i+1})$. We pick an isomorphism $\cL|_{ \sD \times X_{i+1}} \xrightarrow{\sim} \cO$. Under this isomorphism, the $s_j$ correspond to analytic functions $g_j$ on $( \sD \times (X_i \cap X_{i+1})) \cup ( \Dstar \times X_{i+1})$, and it suffices to show that the $g_j$ extend to functions on $ \sD \times X_{i+1}$. But this follows from \autoref{extfunction} below. Therefore, we have that the sections $s_j$ extend to $\sD \times X$, as required. 

\subsubsection*{Step 2: Extending the map.}
It suffices to extend the map $f\vert_{ \Dstar \times X_i} :  \Dstar \times X_i \rightarrow \PP^n$ to an analytic map $\sD \times X_i\rightarrow \PP^n$ for each $i$, and so we suppose that $X = X_1.$ By Step 1, we have functions $(g_j)_{0\leq j \leq n} \in \cO( \sD\times X )$ such that $f(z) = [g_0(z):g_1(z):\hdots g_n(z)]$ for all $z\in  \Dstar \times X$. We may assume that $t\nmid g_{j_0}$ for some $j_0$, where $t$ is the coordinate on $\sD$. We will prove that the $g_j$ have no common zero under this assumption. It suffices to check this at the level of classical points. To that end, let $z = (0,x) \in \{0\} \times X$ for some classical point $x\in X$. As $f$ extends to a map on a neighbourhood of $z$, there exists an affinoid neighbourhood $x\in U_x \subset X$ (and we may further assume that $U_x$ is isomorphic to $\sD^b$) such that $f$ extends to $ \sD \times U_x$. As $\cO( \sD \times U_x)$ is a UFD, there exists a tuple of function $(h_0,\hdots h_n)$ that define the extension of $f$ to $ \sD \times U_x$. By possibly shrinking $U_x$ further, we may assume that some $h_j$ is invertible on $ \sD \times U_x$, say $h_0$. Let $h = \frac{g_0}{h_0} \in \cO( \sD\times U_x)$. As the maps given by the $(g_j)$ and the $(h_j)$ agree on $ \Dstar \times U_x$, we have that $g_j = hh_j$ for every $j$. Therefore, the locus of common zeros in $\sD\times U_x$ of the $g_i$ equals the vanishing locus of $h$. We also have that the vanishing locus of $h$ is contained in the vanishing locus of $t$ (as the $g_j$ do not have a common zero on $ \Dstar \times X$). Therefore, we have that $h \mid t^m$ for some integer $m$. As $t$ is irreducible and therefore prime in $\cO(\sD\times U_x)$, we have that $h = u \cdot t^{m'}$, where $u$ is non-vanishing and $m'\leq m$. If $m'>0$, then we have that $t$ divides $ g_{j_0}|_{ \sD\times U_x}$ and therefore $ t$ divides $g_{j_0}$. This is a contradiction! Therefore, $h$ is non-vanishing and hence $g_0(z) \neq 0$. As $z$ was an arbitrary classical point, we have proved that the functions $g_j$ have no common zero, and therefore $f$ extends to a map $ \sD \times X \rightarrow (\PP^n)^\an$ as required.

\end{proof}

\begin{lemma}\label{extfunction}
    let $X$ denote a smooth connected affinoid variety, and let $U\subset X$ denote a rational open affinoid subdomain. Then, $\cO( \Dstar \times X) \cap \cO(\sD_{\epsilon} \times U) = \cO( \sD \times X)$. 
\end{lemma}
\begin{proof}
    We choose a coordinate $t$ on $\sD$ centered at $0$. Let $g$ be a function in the intersection. As $g\in \cO( \Dstar \times X)$, we may write $g = \sum_{m\in \Z} b_m t^m$, where $b_m\in \cO(X)$ satisfy appropriate convergence conditions. As $g$ extends to a function on $ \sD_{\epsilon}\times U$, we have that $b_m|_{U} \equiv 0$ for $m<0$. By the identity principle, we have that $b_m \equiv 0$ on $X$. The lemma follows.
\end{proof}

\begin{corollary}\label{hyperlocalimpliesgeneraltopologicalcover}
    Let $U$ be a smooth rigid analytic variety over $F$ and $Y \subseteq U$ a closed smooth divisor. Let $Z$ be rigid analytic variety over $F$ and let $W \subseteq (\mathbb{P}^n_F)^\an$ be a locally closed reduced analytic subspace of projective $n$-space. Let $j : Z \rightarrow W$ be a topological covering of rigid analytic spaces. Let $f : U\setminus Y \rightarrow Z$ be an analytic map. 
    Suppose that there exists a rational open neighbourhood $U_z \subset U$ and an analytic extension of $f$ to $U_z$ for every classical point $z\in U\setminus Y.$
    Then $f$ extends analytically to a map $U \rightarrow Z.$
\end{corollary}
\begin{proof}
    It suffices to prove that $(j \circ f) : U\setminus Y \rightarrow W$ extends analytically to a map $g : U \rightarrow W$. To see this, suppose that we had such an extension $g$. Pick an open cover of $W = \cup_s W_s$ by connected affinoid subdomains such that $j^{-1}(W_s) = \coprod_t Z_{st}$ is a disjoint union of connected affinoid subdomains $Z_{st}$ such that $j\vert_{Z_{st}} : Z_{st} \rightarrow W_s$ is an isomorphism. Cover $U_s := g^{-1}(W_s)$ by connected affinoid subdomains $U_s = \cup_r U_{sr}.$ Note that $U_{sr} \setminus Y$ is also then connected, and hence for each $r$ there is a unique $t$ such that $f(U_{sr}\setminus Y) \subseteq Z_{st}$. The map $((j\vert_{Z_{st}})^{-1} \circ g) : U_{sr} \rightarrow Z_{st}$ is then the unique extension of $f$ to $U_{sr}$. These extensions being unique agree on pairwise intersections, and hence necessarily glue to a map $f : U \rightarrow Z.$ 

    By the above paragraph, we may assume that $Z$ is a locally closed analytic subvariety of $(\mathbb{P}^n_F)^\an.$ We may also assume that $U$ is affinoid. Applying Kiehl's tubular neighbourhood theorem \cite[Theorem 1.18]{kiehl-derham} and if need be replacing $U$ by a further affinoid cover, we reduce to the case that $U = \sD \times X$ and $Y = \{0\} \times X$. The corollary then follows from \autoref{hyperlocalextensionforX}.
\end{proof}

We are now ready to treat the case of Shimura varieties of abelian type and $\Ag$. Let $(G,X)$ denote a Shimura datum. Recall that it is of Hodge type if there exists an embedding of $(G,X)$ into the Shimura datum associated to $\GSp(V,\psi)$ as above for some positive integer $g$. A Shimura datum $(G_1,X_1)$ is of abelian type if there exists a Shimura datum $(G,X)$ of Hodge type and a central isogeny $G^{\der}\rightarrow G_1^{\der}$ which induces an isomorphism $(G^{\ad},X^{\ad}) \rightarrow (G_1^{\ad},X_1^{\ad})$ of adjoint data. 

\begin{proposition}\label{prop:reducetoAgpar}
 In order to prove \autoref{main}, it suffices to prove the following result. Let $\AgKpar$ have full level-$\ell$ structure at an odd prime $\ell \neq p$ and let $\Dab \rightarrow \AgKpar^\an$ denote a rigid-analytic map. Then, there exists an $\epsilon >0$ such that $f$ extends analytically to a map $\sD_{\epsilon}^{1+b} \rightarrow ({\AgKpar^{\textrm{BB}}})^{\an}$.  
\end{proposition}

We will need the following lemma. 
\begin{lemma}\label{lem:finiteetale}
    Let $S,T$ denote algebraic varieties over a non-archimedean field $F$, and let $\bar{S},\bar{T}$ denote projective compactifications of $S$ and $T$. Let $\pi : S \rightarrow T$ be a finite \'etale morphism of varieties over the non-archimedean field $F$ that extends to a finite map $\bar{S} \rightarrow \bar{T}$. Then, every analytic map $\Dstar \times X \rightarrow S$ defined over $F$ extends to map $\sD \times X \rightarrow \bar{S}$ if and only if every analytic map $\Dstar\times X \rightarrow T$ extends to a map $\sD \times X \rightarrow \bar{T}$.
\end{lemma}
\begin{proof}
    We first observe that the lemma immediately reduces to replacing $\Dstar \times X \subset \sD \times X$ by $\Dab \subset \sD^{1+b}$. Indeed, \autoref{hyperlocalextensionforX} implies that proving the extension theorem for $\Dstar\times X \subset \sD \times X$ for any smooth $X$ (for both $(S,\bar{S})$ and $(T,\bar{T})$) reduces to proving it for $\Dab \subset \sD^{1+b}$. This is because every classical point $x\in X$ admits a neighbourhood $x\in U_x\in X$ with $U_x$ isomorphic to $ \sD^{\dim X}$. 
    
    Suppose that $(T,\bar{T})$ satisfies the hypothesis of the lemma. 
    Let $f : \Dab \rightarrow S^\an$ be a rigid-analytic map. Then the assumption on $(T,\bar{T})$ implies that $\pi \circ f : \Dab \rightarrow T^\an$ extends to a rigid analytic map $F : \sD^{1+b} \rightarrow \bar{T}^\an.$ This implies that for every $x\in \sD^{1+b}\setminus (\Dab)$, there is a rational open neighbourhood $x\in U_x \subset \sD^{a+b}$ such that $F(U_x) \subset \bar{T}^\an$ is bounded (i.e. contained in an affinoid subdomain of $\bar{T}^\an$). Therefore, it follows that $f(U_x \cap \Dab) \subset \bar{S}^\an$ is also bounded, and therefore (by the Riemann extension theorem) extends to a rigid-analytic map $U_x \rightarrow \bar{S}^\an$. We conclude by applying \autoref{hyperlocalextensionforX}.

    Suppose now that $(S,\bar{S})$ satisfies the hypothesis of the lemma, and that $f : \Dab \rightarrow T^\an$ is an analytic map. Let $x \in \sD^{1+b}\setminus (\Dab)$ be a classical point. By \autoref{hyperlocalextensionforX}, it suffices to find a neighbourhood $x\in U_x \subset \sD^{1+b}$ and a map $U_x \rightarrow \bar{T}$ that extends $f|_{U_x \cap (\Dab)}$. We pick $U_x$ such that the pair $(U_x\cap \Dab) \subset U_x$ is isomorphic to the pair $\sD^{\times}\times \sD^{b} \subset \sD^{1+b}$ with $x$ mapping to the origin. Consider the pullback $\pi'$ of $\pi$ along $f|_{U_x}$. By \cite[Proposition 4.2.1 and Lemma 4.2.2]{DLLZ2}, we may further shrink $U_x$ such that $\pi': S\times_T (U_x\cap (\Dab)) \rightarrow (U_x\cap (\Dab))$ is just a standard Kummer-etale cover $\sD^{\times}\times \sD^{b}\rightarrow \sD^{\times}\times \sD^{b}$ (after possibly taking a finite extension of the base field). Letting $f'$ denote the map $S\times_T U_x \rightarrow S$, we have that that $\hat{f}'$ extends to a map $\sD^{1+b}\rightarrow \bar{S}$. This in turn implies that $f$ is rigid-subanalytic, and therefore by the rigid-subanalytic Riemann extension theorem \cite{p-adic-definable-chow}, we get that $f$ also extends to a map $ U_x \rightarrow \bar{T}^\an$, The lemma follows.    
   
\end{proof}

\begin{proof}[Proof of \autoref{prop:reducetoAgpar}]
    By the same argument as in the first paragraph of \autoref{lem:finiteetale}, we may replace $\Dstar\times X \subset \sD \times X$ by $\Dab \subset \sD^{1+b}$. By \autoref{hyperlocalextensionforX}, it suffices to prove that the map extends to $\sD^{1+b}_{\epsilon}$. Therefore, it suffices to reduce from the case of abelian type to the case of $\AgKpar$ with full level at $\ell$ and maximal parahoric level at $p$. 
    
    We first reduce to the case of $\AgK$ for some choice of level structure. First, let $(G,X)$ denote a Shimura datum of Hodge type and let $S(G,X)_{\sK_G}$ denote the Shimura variety with torsion-free level structure $\sK_G$. Replacing $\sK_G$ by a subgroup has the effect of replacing $S(G,X)_{\sK_G}$ by a finite \'etale cover. Further, this finite \'etale cover induces a finite map for the Baily-Borel compactifications. Therefore, we may replace $\sK_G$ by any finite-index subgroup by \autoref{lem:finiteetale}. The case of Hodge type Shimura varieties now reduces to the case of $\AgK$ (for an approrpiate choice of torsion-free level structure $\sK$) by \cite{DeligneTravaux} and the fact that a map between Shimura varieties extends to their Baily-Borel compactifications. Suppose now that $(G_1,X_1)$ is an abelian type Shimura datum and let $S(G_1,X_1)_{\sK_{G_1}}$ denote the Shimura variety with torsion-free level $\sK_{G_1}$. Then, there exists a connected component of $S(G_1,X_1)_{\sK_{G_1}}$ which is realized as the quotient of a connected component of an appropriate Shimura variety of Hodge type (for example, see \cite[Page 970]{Kisinintegralmodels}) by a finite group that acts freely (and at the cost of replacing $\sK_{G_1}$ by a finite-index subgroup we may assume that the Hodge type Shimura variety has torsion-free level). Therefore, the extension theorem for this connected component of $S(G_1,X_1)_{\sK_{G_1}}$ reduces to the case of Hodge type by again applying \autoref{lem:finiteetale}. Further, the Hecke algebra acts transitively on the set of connected components of $S(G_1,X_1)_{\sK_{G_1}}$ by \'etale correspondences. Therefore, the case of abelian type Shimura varieties reduces to the case of $\AgK$ with torsion-free level structure $\sK$. 

    By another application of \autoref{lem:finiteetale}, we may assume that $\AgK$ has full level-$\ell$ structure. Finally, let $\sKpar = \sK^p(\sKpar)_p$ with $\sK^p$ as above, and $(\sKpar)_p$ some maximal parahoric subgroup containing $\sK_p$. The level structure of $\AgKpar$ is still torsion-free (because of the condition imposed on $\ell$), and the map $\AgK\rightarrow \AgKpar$ is a finite \'etale map. Therefore, \autoref{main} reduces to the case of $\AgKpar$ with full level-$\ell$ structure as required.   

\end{proof}
For the rest of the paper, we will work in the setting of maximal parahoric level structure at $p$ and full level-$\ell$ structure. To further lighten notation, we will drop the subscript $\textrm{par}$ in $\AgKpar$.

\section{The bad reduction locus}\label{sec:badred}
Let $\AgK^{\good} \subset \AgK^{\an}$ denote the good reduction locus. This is just $\AgK^{\rig}$, the rigid-generic fiber of the formal $p$-adic completion of the integral model of $\AgK / \Z_p$, but we will use the notation $\AgK^{\good}$ for this space. The purpose of this section is to reduce the main theorem to the extension property of $\AgK^{\good}$: 
\begin{theorem}\label{thm:reducetocaseofgoodred}
    Suppose that for every $g\geq 1$, every rigid-analytic map $\Dab \rightarrow \AgK^{\good}$ defined over $F$ extends to a map $\sD^{1+b}\rightarrow \AgK^{\good}$. Then, \autoref{main} is true. 
\end{theorem}
\subsection{Constancy of reduction type}\label{sec:allgoodallbad}

\begin{definition} 
Let $A/F$ be an abelian variety with semistable reduction. The abelian rank of $A$ is defined to be the dimension of the maximal abelian quotient of the the identity component of the special fiber of the N\'eron model of $A$ over the ring of integers $\mathcal{O}_F$ of $F$.
\end{definition}

We prove the following result in this subsection. 
\begin{theorem}\label{badisfullybad}
Let $f:\Dstar \times \sD^b \rightarrow \AgK^{\an}$ denote a rigid-analytic map defined over $F$. Then, the abelian rank of $x^*\Aguni$ is constant for all classical (rigid) points $x \in \Dstar \times \sD^b$.
In particular, the image of $f$ is either entirely contained in the good-reduction locus, or entirely contained in the bad-reduction locus. 
\end{theorem}

This result easily reduces to the one-dimensional case, that is to the case where $b =0$. Furthermore, we can reduce further to the case where the source of the rigid-analytic map $f$ is a closed rigid-analytic annulus $\Ann$ defined over $F$. The above result then follows immediately from \autoref{badannulusisfullybad} below. 
\begin{theorem}\label{badannulusisfullybad}
Let $f:\Ann \rightarrow \AgK^\an$ be a rigid-analytic map defined over $F$. Then, the abelian rank of $x^* \Aguni$ is constant for all classical points $x\in \Ann(L)$ for all finite extensions $L$ of $F$. 
In particular, the image of $f$ is either entirely contained in the good reduction locus, or entirely contained in the bad reduction locus. 
\end{theorem}

Before proving of these results, we will first describe the Baily-Borel compactification $\AgKbb$ of $\AgK$. The boundary, $\AgKbb \setminus \AgK$, has a natural stratification, and each stratum is canonically isomorphic to $\mathcal{A}_{g'\mathsf{K}'}$, for $0\leq g' < g$. The level structure $\sK'$ inherits our initial assumptions on the level structure on $\AgK$ (namely, maximal parahoric at $p$, and full level structure at an odd prime $\ell \neq p$). For brevity, we denote this stratum by $\AgKbb(g')$. We also have that each $\AgKbb(g')$ is locally closed in $\AgKbb$, its closure is  the union of strata $\AgKbb(g'')$ for each $0\leq g'' \leq g'$. The above stratification exists integrally over $\mathbb Z_p$.
Finally, the abelian rank of the mod $\pi$ reduction of an abelian variety over $F$ can be easily read off using the mod $p$ Baily-Borel compactification. Indeed, let $x\in \AgK(F)$ denote a point which specializes a point in $\AgKbb(g')(\F)$. Then, the reduction of the abelian variety parameterized by $x$ has abelian rank $g'$. We make the following definition for brevity: 
\begin{definition}\label{abelianrankdefn}
    Let $x\in \AgK(F)$ be a point, and let $A_x$ denote the abelian variety parameterized by $x$. We define the abelian rank of $x$ to be the abelian rank of the mod $\pi$ reduction of $A$. 
\end{definition}

\subsubsection{Thin annuli}
We will first prove \autoref{badannulusisfullybad} for \emph{thin annuli}, i.e. when the inner and outer radii are equal. Our proof builds on \cite[Theorem 4.1]{PST}. The main results for this case will be \autoref{galthinannulus} and \autoref{cor:extendstoGm} below, from which we will deduce the constancy of abelian ranks. Given any variety (or rigid-analytic variety) defined over $F$ and given any classical $L$-point of this variety (or rigid-analytic variety) $x$ where $L/F$ is a finite extension, we will let $\bar{x}$ denote the geometric point induced by $x$.

\begin{proposition}\label{galthinannulus}
   Let $\Ann_0 := \Gm^{\rig}$
    denote the thin annulus over $F$. Let $\mathbb{L}$ be an \'etale $\mathbb Z_\ell$-local system on $\Ann_0$ for a prime $\ell \neq p$. Suppose that there exists a classical $F$-point $x\in \Ann_0(F)$ such that the Galois representation of $\gal_F$ induced by $\mathbb{L}_{\bar{x}}$ is unramified. Then $\mathbb{L}$ is unramified everywhere; i.e.
    given any finite extension $L$ of $F$ and any classical $L$-point $y\in \Ann_0(L)$, the $\ell$-adic Galois representation of $\gal_L$ induced by $\mathbb{L}_{\bar{y}}$ is unramified.    
\end{proposition} 

Firstly, in order to prove \autoref{galthinannulus} we may assume that $F = \Breve{F}$, the completion of the maximal unramified extension of $F$ inside $\overline{F}$. 

Any $\ell$-adic local system on $\Gm/\cO_F$ can be pulled back to give a local system on $\Ann_0$. The main step in the proof of \autoref{galthinannulus} is the following result.
\begin{proposition}\label{cor:extendstoGm}
Let $\mathbb{L}$ denote an \'etale $\mathbb Z_\ell$-local system on $\Ann_0$ with prime-to-$p$ monodromy. Suppose there exists a classical point $x\in \Ann_0(F)$ such that $\mathbb{L}_x$ is trivial. Then, $\mathbb{L}$ is pulled back from a local system on $\Gm/\cO_F$. 
\end{proposition}

The proofs of \autoref{galthinannulus} and \autoref{cor:extendstoGm} will be entirely $\ell$-adic, so we first start by recalling some results about the prime-to-$p$ \'etale fundamental groups of thin annuli. Let $\pi'_1$ denote the maximal prime-to-$p$ quotient of the \'etale fundamental group. 

\begin{lemma}\label{templempione}
Let $x\in \Ann_0(F)$ be a classical point. Let $x$ also denote the induced algebraic point of the variety ${\Gm}_F$. The canonical map $\pi'_1(\Ann_0,\bar{x}) \rightarrow \pi'_1({\Gm}_F,\bar{x})$ is an isomorphism.

\end{lemma}
\begin{proof}
The canonical map $\pi'_1(\Ann_0,\bar{x}) \rightarrow \pi'_1({\Gm}_F,\bar{x})$ factors as $$\pi'_1(\Ann_0,\bar{x}) \rightarrow \pi'_1(\Gm^{\an},\bar{x}) \rightarrow \pi'_1({\Gm}_F,\bar{x}).$$ 
It suffices to prove that both the above maps are isomorphisms. By \cite[Theorem 3.1]{lutkebohmert-riemann}, we have that every finite \'etale cover of $\Gm^{\an}$ is uniquely algebraizable, and therefore the second map is an isomorphism. It suffices to prove that the first is an isomorphism. 

By \cite[Theorem 2.11 (b)]{lutkebohmert-riemann}, the geometric prime-to-$p$ \'etale covers of $\Ann_0$ are given by adjoining prime-to-$p$ roots of $t$. This is true for ${\Gm}_{\bar{F}}$, and therefore also for ${\Gm^{\an}}_{\bar{F}}$ (again by \cite[Theorem 3.1]{lutkebohmert-riemann}). 
Therefore, the canonical map $\pi'_{1,\geom}(\Ann_0,\bar{x}) \rightarrow \pi'_{1,\geom}(\Gm^{\an},\bar{x})$ is an isomorphism. As the arithmetic-geometric exact sequences of fundamental groups are functorial for maps between rigid-analytic varieties, the map $\pi'_1(\Ann_0,\bar{x}) \simeq \pi'_1(\Gm^{\an},\bar{x})$ is an isomorphism, and the result follows. 
\end{proof}

We are now ready to prove \autoref{cor:extendstoGm} and \autoref{galthinannulus}.
\begin{proof}[Proof of \autoref{cor:extendstoGm}]
By \autoref{templempione}, $\mathbb{L}/\Ann_0$ is induced by a local system on ${\Gm}_F$ (which we shall also denote by $\mathbb{L}$). In order to prove that $\mathbb{L}/{\Gm}_F$ is pulled back from a local system on ${\Gm}_{\cO_F}$, it suffices to prove that the image $J$ of the map $\gal'(\bar{F}/F)\rightarrow \pi'_1({\Gm}_F,\bar{x})$ induced by the point $x\in \Gm(F)$ equals the kernel of the canonical map $\pi'_1({\Gm}_F,\bar{x})\rightarrow \pi'_1({\Gm}_{\cO_F},\bar{x})$. That $J$ is contained in the kernel follows from the fact that $x$ is induced by a $\cO_F$-point of ${\Gm}_{\cO_F}$, and $F$ has algebraically closed residue field. 
Every element of $\pi'_1({\Gm}_F,\bar{x})$ can be uniquely written as $g\cdot \alpha_{\geom}$, where $g\in J$ and $\alpha_{\geom} \in \pi'_{1,\geom}({\Gm}_F,\bar{x})$. But the composite map $\pi'_{1,\geom}({\Gm}_F,\bar{x})\rightarrow \pi'_1({\Gm}_F,\bar{x}) \rightarrow \pi'_1({\Gm}_{\cO_F},\bar{x})$ is an isomorphism. Therefore, $J$ is indeed equal to the kernel of the map $\pi'_1({\Gm}_F,\bar{x})\rightarrow \pi'_1({\Gm}_{\cO_F},\bar{x})$. The result follows. 
\end{proof}

\begin{proof}[Proof of \autoref{galthinannulus}]
By \autoref{templempione}, we may work in the setting of $\mathbb{L}/{\Gm}_{F}$, with the hypothesis that $\mathbb{L}_x$ is constant for a point $x$ induced by a $\cO_F$-valued point of ${\Gm}_{\cO_F}$. We must prove that $\mathbb{L}_y$ is constant for every $y$ induced by a $\cO_L$-valued point of ${\Gm}_{\cO_F}$ for every finite extension $L/F$. \autoref{cor:extendstoGm} yields that $\mathbb{L}$ descends to $\Gm$. It then follows that the Galois representation of $\gal(\bar{L}/L)$ induced by $y$ factors through the \'etale fundamental group of $\cO_L$. The result follows as we're in the case of algebraically closed residue fields. 
\end{proof}

We are now ready to prove our results for thin annuli. 

\begin{proof}[Proof of \autoref{badannulusisfullybad} for thin annuli]
    We will again assume that $F = \breve{F}$. Let $A/F$ denote an abelian variety with semistable reduction. By \cite{Grothendieck}, the abelian rank of $A$ determines and is fully determined by the dimension $d$ of inertial invariants of $T_{\ell}(A)$. Indeed, if $r$ is the abelian rank of $A$, the dimension of inertial invariants equals $d = g+r$.
    
    Let $\mathbb{L}/\Ann_0$ denote the family of $\ell$-adic Tate modules induced by $\Aguni$. Recall that $\Aguni[\ell]$ is the constant group scheme $(\Z/\ell\Z)^{2g}$, and so $\mathbb{L}$ has prime-to-$p$ monodromy. By the above remark, it suffices to prove that the dimension of inertial invariants of $\mathbb{L}_x$ is independent of the choice of classical point $x \in \Ann_0$. To that end, let $x\in \Ann_0$ denote the classical point with maximal inertial invariants, and let $\mathbb{L}^I_{\bar{x}} \subset \mathbb{L}_{\bar{x}}$ denote the space of inertial invariants (as we are in the setting of algebraically closed residue fields, this is the same as the space of Galois invariants). By \autoref{templempione}, we may work in the setting of $\mathbb{L}/{\Gm}_F$. Consider the image $J$ of the map $\gal'(\bar{F}/F)\rightarrow \pi'_1({\Gm}_F,\bar{x})$. As in the proof of \autoref{templempione}, $J$ is a normal subgroup of $\pi'_1({\Gm}_F,\bar{x})$, and therefore $\mathbb{L}^I_{\bar{x}}$ is preserved by the action of $\pi'_1({\Gm}_F,\bar{x})$. It follows that $\mathbb{L}^I_{\bar{x}}$ is the fiber at $x$ of a sub local system $\mathbb{L}'$ of $\mathbb{L}$. Restricting $\mathbb{L}'$ to $\Ann_0$, we may now apply \autoref{galthinannulus} to deduce that $\mathbb{L}'$ is unramified everywhere. Therefore, the rank of the space of inertial invariants $\mathbb{L}^{I}_{\bar{y}}$ is greater than or equal the rank of $\mathbb{L}'$, and the assumption on $x$ implies that these two ranks are equal, as required.

\end{proof}

\begin{remark}
 Lutkebohmert also proves that the prime-to-$p$ \'etale fundamental group of closed thick annuli $\Ann_m = \Sp(F\langle x,y \rangle / (xy - \pi^m))$ for $m\geq 1$ is the same as the thin (and algebraic) case. However, this doesn't suffice to prove Theorem \ref{badannulusisfullybad} for thick annuli. In order for the same argument to go through, we would need to replace ${\Gm}_{\cO_F}$ with the integral model $\Spf \cO_F\langle x,y \rangle / (xy - \pi^m))$, and prove that $\mathbb{L}$ is induced by a local system on $\Spf \cO_F\langle x,y \rangle / (xy - \pi^m))$. Unlike in the case $m=0$, the natural map ${\pi'}_{1,\geom}(\Ann_m,\bar{x}) \rightarrow \pi'_1 \Spf \cO_F\langle x,y \rangle / (xy - \pi^m))$ is not an isomorphism! Indeed, it is easy to construct $\ell$-adic local systems on $\Ann_m$ which are ramified at certain points but unramified at others, in stark contrast to the case of $\Ann_0$. In order to prove \autoref{badannulusisfullybad} in the case of thick annuli (which we do below in Section \ref{subsubthick}), it will not suffice to work solely with $\ell$-adic local systems -- we will make strong use of the fact that the local systems in question are induced by families of abelian varieties on $\Ann_m$. Whereas, our proof of \autoref{badannulusisfullybad} for thin annuli follows (after using \cite{Grothendieck}) directly from \autoref{galthinannulus}, which is a statement purely about $\ell$-adic local systems on thin annuli. 
\end{remark}

\subsubsection{Thick annuli}\label{subsubthick}

\begin{lemma}\label{lem:zariski-tube-quasicompact}
    Suppose $\Ann = (\Spf \cO_F [x,y]/( xy - \pi^n ))^{\rig}$ is a rigid-analytic annulus over $F$ and $\mathfrak{A}$ is an admissible formal $\cO_{{F}}$-model of $\Ann$ as found in \autoref{prop:specialfiber}. Let $W \subseteq \mathfrak{A}\otimes_{\cO_{{F}}} \mathbb{F}$ be a Zariski-open subset of the special fiber. Then the tube over $W$ inside $\Ann$ is a quasi-compact rational open subvariety of $\Ann$.   
\end{lemma}
\begin{proof}
   We may assume that $F = \Breve{F}$. Let $\mathfrak{A} = \cup_{i=1}^m U_i$ be a finite formal affine cover of $\mathfrak{A}$. It suffices to prove that the tube inside $U_i^{\rig}$ over $W\cap (U_i\otimes_{\cO_{\Breve{F}}} \overline{\F})$ is quasi-compact. We may thus replace  $\mathfrak{A}$ by $U_i$ and assume that $\mathfrak{A}$ is a formal affine admissible $\cO_{\Breve{F}}$-scheme, $\mathfrak{A} = \Spf(B)$. The special fiber $\Spec(B \otimes \overline{\F})$ is a finite type $\overline{\F}$-scheme of dimension one such that its irreducible components (with their reduced structures) are isomorphic to open subschemes of $\mathbb{A}^1_{\overline{\F}}$.
   The complement $(\mathfrak{A}\otimes \overline{\F})\setminus W$ is a finite set of closed points of $\Spec(B\otimes_{\cO_F} \overline{\F})$ and we may therefore find an element $b \in B$ such that its reduction $\overline{b} \in B\otimes \overline{\F}$ cuts out this finite set of closed points. 
   The rigid generic fiber of the principal open subset $\Spf(B\langle b^{-1}\rangle) \subseteq \Spf(B)$ given by the $\pi$-adic completion of $B[b^{-1}]$ is then exactly the tube inside $\Sp(B\otimes_{\cO_F} \Breve{F})$ over $W$, and is thus affinoid, in particular quasi-compact.
\end{proof}

\begin{corollary}\label{cor:goodbaddescrip}
    Let $\Ann$ be a closed annulus as above and let $f : \Ann \rightarrow \AgK^\an$ be a rigid-analytic map. Then the locus of points in $\Ann$ having maximal abelian rank is a finite union of closed subannuli contained in $\Ann$. 
\end{corollary}
\begin{proof}
    Let $V \subseteq \Ann$ denote the locus of points of $\Ann$ with maximal abelian rank.
    By \autoref{lem:zariski-tube-quasicompact}, $V$ is a quasi-compact rational open of $\Ann$. 

    Passing to $\C_p$, the quasi-compact rational open $V_{\C_p} := V \hat{\otimes} \C_p \subseteq \Ann \hat{\otimes} \C_p$ is a finite union of rational subdomains, and hence is definable over $\C_p$ in the rigid-subanalytic language $L_{an}^D$ of Lipshitz--Robinson \cite{lipshitz-robinson-rigid-subanalytic}.  
    Working at the level of classical points, \cite[Theorem 3.8.2]{lipshitz-robinson-rigid-subanalytic} yields that $|V_{\C_p}| \subseteq |\C_p^\times|$ is quantifier-free definable in the divisible ordered abelian group $|\C_p^\times|$. Hence, $|V_{\C_p}|$ is a finite disjoint union of closed intervals in $|\C_p^\times|.$

    From \autoref{badannulusisfullybad}, it follows that for each $\gamma \in |\C_p^\times|$ with $|\pi|^n \leq \gamma \leq 1$, the thin subannulus of $\Ann$ (defined over an appropriate finite extension of $F$) at radius $\gamma$ is either completely contained in $V_{\C_p}$ or completely contained in its complement. 
    The corollary now follows from the conclusion of the previous paragraph.
\end{proof}

\begin{proof}[Proof of \autoref{badannulusisfullybad}]
By \autoref{cor:goodbaddescrip}, the locus of points in $\Ann$ having maximal abelian rank is a finite union of closed subannuli of $\Ann.$ Suppose that for the sake of a contradiction, this locus of points is not all of $\Ann.$ By replacing $\Ann$ by a smaller closed annulus, we may assume that the outer thin sub-annulus $\Ann_0\subset \Ann$ is the locus of points of $\Ann$ with maximal abelian rank, say $g'$. Let $\Ann_1\subset \Ann$ denote the inner thin sub-annulus, and let $\mathsf{N}$ denote the complement in $\Ann$ of $\Ann_0$ and $\Ann_1$. The map $\Ann \rightarrow \AgK$ is induced by a map of integral models $\fA \rightarrow \AgKbb$. 

Recall the initial integral model of $\Ann$, given by $\Spf \cO_F \langle s,t \rangle/(st - \pi^n )$ (for some $n$). Let $X$ denote the closed subscheme $s = 0,\pi =0$, and $Y$ denote the closed subscheme $t = 0,\pi = 0$, and let $o$ be the node $X\cap Y$. Without loss of generality, suppose that $\Ann_0$ is the tube over $X\setminus \{o\}$. Now, consider the integral model $\fA$ of $\Ann$ constructed in the appendix (\autoref{prop:specialfiber}), and the map $\beta: \fA \rightarrow \Spf \cO_F \langle s,t \rangle/( st - \pi^n )$. Recall that the special fiber of $\fA$ contains two distinguished irreducible components $L$ and $M$, both isomorphic to $\mathbb{A}^1$, which map onto $X$ and $Y$ respectively. By \autoref{anandremark}, only one point $x\in L$ (and $y\in M$) maps to the node $o\in X$. Therefore, the preimage of any point $z\in L\setminus \{x\}$ under the specialization map (via the integral model $\fA$) is a subset of $\Ann_0$. This discussion is captured by \autoref{picture}. 

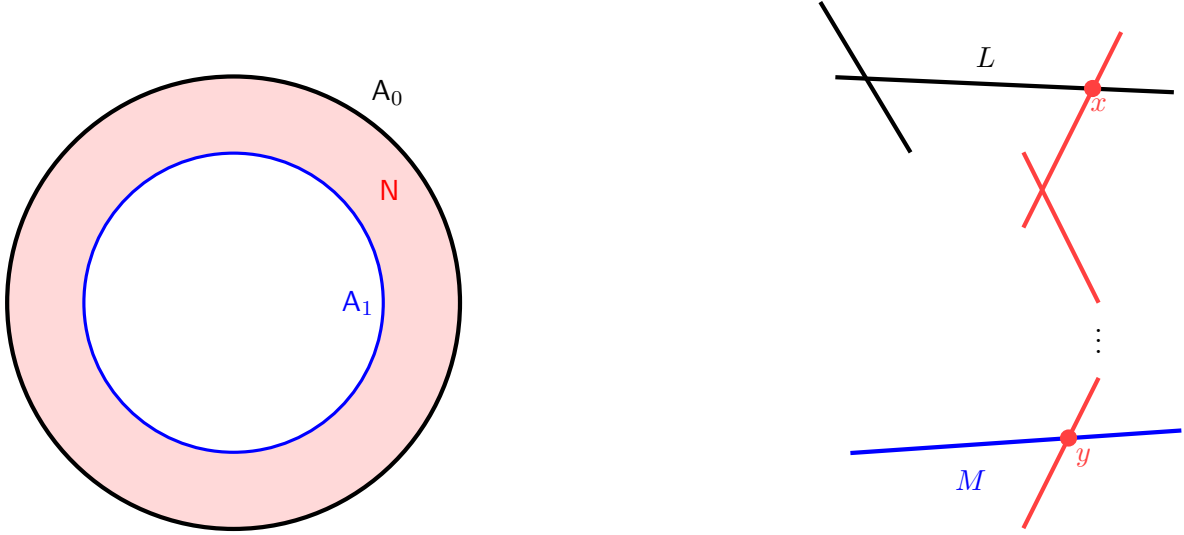
\begin{figure}

\begin{tikzpicture}

    \draw[fill=red!15,color=red!15] (-1.5,-3) circle [radius=3]; 
    \draw[fill,color=white] (-1.5,-3) circle [radius = 2]; 
   \node[right] at (.3,-1.5) {\color{red} $\mathsf{N}$};

\draw[color=black, ultra thick] (-1.5,-3) circle [radius = 3.01]; 
    \node[right] at (.2,-0.2) {\color{black} $\Ann_0$};

    \draw[color=blue, very thick] (-1.5,-3) circle [radius = 1.99]; 
    \node[left] at (.5,-3) {\color{blue}$\Ann_1$};
    
    \draw[color=black, ultra thick] (6.5,0) -- (11,-.20); 
    \node[above] at (8.5,0) {\color{black} $L$};
    \draw[color=black, ultra thick] (7.5,-1) -- (6.3,1);
    
    \draw[color=red!75, ultra thick] (10.3,0.6) -- (9,-2); 
    \draw[fill,color=red!75, thick] (9.92,-0.15) circle [radius = .1];
    \node[below] at (10,-0.15) {\color{red!75} $x$};
    
    \draw[color=red!75, ultra thick] (9,-1) -- (10,-3); 
    \node[below] at (10,-3) {$\vdots$}; 
    \draw[color = red!75, ultra thick] (10,-4) -- (9,-6);

    \draw[color = blue, ultra thick] (6.7,-5) -- (11.1,-4.7);
    \node[below] at (8.3,-5.1) {\color{blue} $M$};
    \draw[fill,color=red!75, thick] (9.6,-4.8) circle [radius = .1];
    \node[below] at (9.8,-4.8) {\color{red!75} $y$};

\end{tikzpicture}
\caption{Generic and special fibers of $\fA$}
\label{picture}
\end{figure}
With this in hand, we have that $L\setminus \{x\}$ maps to $\AgKbb(g') \otimes \F$, while the point $x$ maps to $\AgKbb(g'')\otimes \F$ for some $g'' < g'$. Recall that $\AgKbb(g')$ is canonically isomorphic to $\mathcal{A}_{g',\sK'}$, where $\sK'$ is such that there is full level $\ell$ structure. Therefore, our initial assumption that the locus of points in $\Ann$ having maximal abelian rank is not equal to all of $\Ann$ implies that there exists an algebraic map $L \setminus \{x\}\rightarrow \mathcal{A}_{g',\sK'}$ where $x$ maps to a boundary-point of $\mathcal{A}_{g',\sK'}$. By \autoref{newlemma} below, this isn't possible, and this contradicts our initial assumption that the locus of points with maximal abelian rank is not all of $\Ann$. The theorem follows.

\end{proof}

\begin{lemma}\label{newlemma}
    Let $g\geq 1$ be any integer. 
    Let $\overline{\AgK}$ be a separated scheme of finite type containing $\AgK\otimes\mathbb F_p$ as an open dense subset, where $\AgK$ has full level-$\ell$ structure. Let $L=\mathbb{A}^1$ and let $L \to \overline{\AgK}$ be a morphism over $\Fpbar$ such that the preimage of $\AgK$ contains $\Gm$. Then the image of $L$ is contained in $\AgK$.
\end{lemma}
\begin{proof}
    The map $L \rightarrow \overline{\AgK}$ must descend to $\F_q$ for some $q$, and so we replace $\Fpbar$ by $\F_q$. The moduli space $\AgK$ is fine, as we have full level-$\ell$ structure. Pull back the universal abelian scheme to $\Gm$. The abelian scheme has semistable reduction at 0, again because we have full level-$\ell$ structure. Therefore, the abelian scheme must have unipotent $\ell$-adic monodromy around 0. 

    The prime-to-$p$ fundamental group of $\Gm$ is solvable. The $\ell$-adic monodromy of the abelian scheme is semisimple, and therefore must be toric. The monodromy around 0, which is both unipotent and toric, therefore must be trivial. It follows that the abelian scheme has good reduction at 0. Therefore, 0 has to map to $\AgK \subset \overline{\AgK}$, as required. 
\end{proof}

\subsection{Reduction to the case of \texorpdfstring{$\AgK^{\good}$}{AgKgood}}
Let $\AgKtor$ denote a toroidal compactification associated to an appropriately fine cone decomposition. There is a canonical map $\phi: \AgKtor \rightarrow \AgKbb$. Recall that $\AgKbb(g') \subset \AgKbb$ (over $\mathbb Z_p$) is the boundary stratum in the Baily-Borel compactification which parameterizes $g'$-dimensional abelian varieties, where $0\leq g' < g$. 
Recall that $\AgKbb(g')$ is isomorphic to the moduli space $\mathcal{A}_{g',\mathsf{K'}}$ of $r$-dimensional abelian varieties with $\mathsf{K'}$ such that there is full level $\ell$ structure.  Let $T_{g'} \subset \AgKtor$ denote the pre-image of $\AgKbb(g')$ under $\phi$. Consider $(\AgKtor)^{T_{g'}}$, the formal completion of $\AgKtor$ along $T_{g'} \times \F$. 

There is a canonical morphism (see for instance \cite[Section 2]{KaiWen}, and also \cite[Theorem 2.5.9]{CaraianiScholze}) of formal schemes $\lambda: (\AgKtor)^{T_{g'}} \rightarrow \widehat{\mathcal{A}}_{g',\mathsf{K'}}$, where $\widehat{\mathcal{A}}_{g',\mathsf{K'}}$ denotes the $p$-adic formal completion of $\mathcal{A}_{g',\mathsf{K'}}$. The canonicity of this morphism yields that the following diagram commutes.

    \begin{equation}\label{badreddiagram}
    \begin{tikzcd}
        (\AgKtor)^{T_g'} \arrow[r, "\lambda"] \arrow[d, "\textrm{sp}"] & \widehat{\mathcal{A}}_{g',\mathsf{K'}} \arrow[d, "\textrm{sp}"] \\
        T_{g'} \otimes \F  \arrow[r, "\phi"] & \mathcal{A}_{g',\mathsf{K}'}\otimes \F
    \end{tikzcd}
    \end{equation}
Here, $\textrm{sp}$ is the specialization map. (Note that $\lambda$ is not the formal completion of the map $\phi$.)
\begin{remark}
    Note that every classical $F$-valued point $x$ in the rigid generic fiber of $(\AgKtor)^{T_{g'}}$ gives a semiabelian scheme over $\cO_F$. The generic fibre of this semi-abelian scheme is analytically uniformized by a semi-abelian variety over $F$, whose abelian part is $g'$-dimensional and has good reduction mod $\pi$. Therefore, canonically associated to $x$ is a $g'$-dimensional abelian variety with good reduction. It is this fact that induces the map $(\AgKtor)^{T_{g'}} \rightarrow \widehat{\mathcal{A}}_{g',\mathsf{K'}}$
\end{remark}

We are now ready to prove the main result of this section. 

\begin{proof}[Proof of \autoref{thm:reducetocaseofgoodred}]
     By \autoref{badannulusisfullybad}, we have that the image of $f$ is contained in the tube over $B_r$ for some Baily-Borel stratum $B_r$. By using \autoref{prop:reducetoAgpar}, it suffices to prove (under the hypothesis of \autoref{thm:reducetocaseofgoodred}) that the map extends after replacing $\Dab$ by $\Dstar_{\epsilon} \times \sD_{\epsilon}^b$ (for some $\epsilon > 0$). Therefore, by the Riemann extension theorem, it suffices to prove that $f(\Dstar_{\epsilon} \times \sD_{\epsilon}^b)$ is contained in the tube over a closed point of $\AgKbb(g')(\Fpbar)$. 
    
    As the image of $f$ avoids the boundary in characteristic 0, we may consider $f$ to be a map $\Dab \rightarrow \AgKtor$ -- note that the image of $f$ must be contained in the rigid generic fiber of $(\AgKtor)^{T_{g'}}$ by \autoref{badisfullybad}. By the commutativity of Diagram \ref{badreddiagram}, it suffices to prove that the map $\textrm{sp}\circ \lambda \circ f$ is constant (after replacing $\Dab$ by ${\Dstar}_{\epsilon}^a \times \sD_{\epsilon}^b$ for some $\epsilon > 0$). However, our hypothesis is that the extension theorem is for $\AgK^{\good}$, and therefore is true for $\mathcal{A}_{g',\sK'}$ for every $g'< g$. Therefore, the map $\lambda \circ f$ extends to a map from $\sD^{a+b}\rightarrow \mathcal{A}_{g',\sf{K}'}$. Therefore, after replacing $\Dab$ by $\Dstar_{\epsilon} \times \sD_{\epsilon}^b$, we have that the image of $\lambda\circ f$, and therefore of $f$, must be contained in the residue space over some closed $x\in \AgKbb(\Fpbar)$. The theorem follows.
 
\end{proof}

\section{Lifts to Rapoport--Zink spaces}\label{sec:lifttoRZ}

The main result of this section is the following theorem. 
\begin{theorem}\label{thm:maptoRZ}
Let $f: \Dab \rightarrow (\AgK \otimes_{\Z_p} F)^{\an}$ be a rigid-analytic map over $F$ such that the image of $f$ is contained in $\AgK^{\good}$. Then, there exists a lift  $\tilde{f}_{\Breve{F}}: \Dab \rightarrow \RZ_x^{\rig}$ of $f_{\Breve{F}}$, where $x \in \AgK(\Fpbar)$ is some fixed point whose residue disk has non-empty intersection with $f(\Dab)$.
\end{theorem}

We recall that we have full $\ell$-structure for some odd prime $\ell \neq p$, and parahoric level structure at $p$. Our method of proof will be to first prove the one-dimensional case, i.e. for maps from $\Dstar$, and then to deduce the general case.  

\subsection{Preliminaries on abelian varieties}

Let $C$ denote a connected algebraic curve over $\overline{\F}_p$, whose geometric, irreducible components (with their reduced induced structure) are isomorphic to either $\PP^1$ or $\A^1$. Further, we assume that the intersection graph of $C$ is a tree. Let $\mathscr{A}$ denote an abelian scheme over $C$ endowed with a polarization $\lambda_{\mathscr{A}} : \mathscr{A} \rightarrow \mathscr{A}^\vee$. Let $x_0 \in C(\overline{\F}_p)$ be a fixed point, and let $\mathscr{A}_0$ denote the fiber of $\mathscr{A}$ at $x_0$. We also assume that there is an isomorphism $\mathscr{A}[\ell] \simeq (\Z/\ell \Z)^{2g}$. 
\begin{definition}\label{ppowerquasiisogeny}
Let $\mathscr{B}, \mathscr{B}'$ be abelian schemes over a base scheme $S$ that is of finite type over $\overline{\mathbb{F}}_p$. 
\begin{enumerate}
    \item We say that an element
$\rho \in \Hom_{\mathrm{Ab.Sch}/S}(\mathscr{B}, \mathscr{B}')\otimes_\Z \Q$,
is a \emph{$p$-power quasi-isogeny over $S$} if there exists some integer $n$ such that $p^n \rho : \mathscr{B} \rightarrow \mathscr{B}'$ is a $p$-power isogeny over $S$.
    \item Suppose $\lambda$ (resp. $\lambda'$) denotes a polarization on the $S$-abelian scheme $\mathscr{B}$ (resp. $\mathscr{B}'$), then we say that a $p$-power quasi isogeny $\rho$ from $\mathscr{B}$ to $\mathscr{B}'$ is \emph{a polarized $p$-power quasi-isogeny over $S$} if there is an integer $n$ such that $\rho^\ast(\lambda') = n \lambda.$
\end{enumerate}
\end{definition}

\begin{proposition}\label{prop:isogenyovertreeofcurves}
Maintain the above setting, and suppose that $B$ is an abelian variety over $\overline{\F}_p$, endowed with a polarization $\lambda_B : B \rightarrow B^\vee$. Let $j_0:  B \rightarrow \mathscr{A}_{x_0}$ denote any polarized $p$-power quasi-isogeny (i.e. $j_0^*(\lambda_{\mathscr{A}_{x_0}}) = n \lambda_{B}$ for some $n\in \Z$). Then, there exists a unique polarized $p$-power quasi-isogeny $j: B\times_{\Fpbar} C \rightarrow \abvar$, which specializes to $j_0$. 
\end{proposition}

\begin{lemma}\label{TateA1}
Let $C$ be either $\A^1$ or $\P^1$, and let $\phi: C\rightarrow \AgK \otimes \Fpbar$ denote some map defined over $\Fpbar$. Let $\mathscr{A}/C$ denote the abelian scheme pulled back from $\AgK$. Then, $\mathscr{A}/C$ is $p$-power isogenous to the constant family $\mathscr{A}_0\times C$, where $\mathscr{A}_0$ is the fiber of $\mathscr{A}$ over some point $x_{0} \in C(\Fpbar)$. 
\end{lemma}
\begin{proof}
By \cite[Proposition 4.4]{Grothabelian}, it suffices to prove that the family of $\ell$-adic Tate modules on $C$ has trivial (geometric) monodromy. This is immediate for $C= \PP^1$, as $\PP^1$ has trivial  geometric \'etale fundamental group. Suppose that $C = \A^1$. In this case, the prime-to-$p$ geometric \'etale fundamental group is trivial. We also have that $\mathscr{A}[\ell]$ is trivial by assumption. Therefore, the family of $\ell$-adic Tate modules on $C$ has trivial monodromy, and the lemma follows.

\end{proof}

\begin{lemma}\label{lem:isogovercurve}
Let $C$ be either $\A^1$ or $\P^1$, $\abvar/C$ a family of polarized abelian schemes with $\abvar[\ell] \simeq \Z/(\ell \Z)^{2g} \times C$, and $x_0\in C$. Then, any polarized $p$-power quasi-isogeny $j_0:  B \rightarrow \abvar_0$ extends uniquely to a polarized $p$-power quasi-isogeny $j:  B \times_{\Fpbar} C \rightarrow \abvar$. 
\end{lemma}
\begin{proof}
 It suffices to prove that the restriction map $\Hom(B\times_{\Fpbar} C,\abvar)\to \Hom(B,\abvar_0)$ has $p$-power index. We claim that there exists a $p$-power isogeny $i: \abvar\rightarrow B\times_{\Fpbar} C$. Indeed, by assumption there exists a $p$-power isogeny $B\rightarrow \abvar_0$ and therefore $\abvar_0 \rightarrow B$, and by \autoref{TateA1} there exists a $p$-power isogeny $\abvar\rightarrow \abvar_0\times C$, whence follows the existence of $i$. 
 
 Post-composition with $i$ defines an  injective map $\Hom(B \times_{\Fpbar} C,\abvar) \rightarrow \End(B \times_{\Fpbar} C)$. There is also an injective map $\Hom(B,\abvar_0) \rightarrow \End(B)$. We have that $\End(B) = \End(B\times C)$ (this is because $\End B \subset  \End(B\times C) = \End(B\times k(C)) \subset \End(B\times \overline{k(C)}) = \End(B) $, where the equality $\End(B\times C) = \End(B\times k(C))$ follows from \cite{FaltingsChai}). Therefore, it suffices to prove that the image of $\Hom(B\times_{\F} C,\abvar)$ in $\End(B\times_{\F} C)$ has $p$-power index. But this is true, as the map $i$ is a $p$-power isogeny, and therefore any endomorphism of $B$ whose kernel contains the kernel of $i$, a $p$-power group, must lie in the image. 
 
 We now prove that the quasi-isogeny is compatible with polarizations. Without loss of generality, we assume that $j$ (and therefore also $j_0$) is an isogeny. Consider the two polarizations $\lambda_{B}$ and $j^*(\lambda_{\abvar})$ on $B\times C$. We view polarizations as maps of abelian schemes between $B\times C$ and its dual $B^{\vee} \times C$. By assumption, there exists an integer $n$ such that $n\lambda_{B} - j^*(\lambda_{\abvar})$ restricts (by specializing to $x_0$) to the zero map between $B$ and $B^{\vee}$. The result follows from the fact that the specialization map is injective. 
\end{proof}

\begin{proof}[Proof of \autoref{prop:isogenyovertreeofcurves}]
This follows directly from the fact that the intersection graph of $C$ is a tree, and \autoref{lem:isogovercurve} above. 
\end{proof}

\subsection{Proof of \autoref{thm:maptoRZ}}
\label{sec:proofoflifttoRZ}

\begin{proof}

We first treat the case of $\Dstar$. Write $\Dstar = \bigcup_{n\geq 0} \Ann_n$ as an increasing union of closed annuli, and let $f_n = f|_{\Ann_n}$. Without loss of generality, let $\tilde{x} \in \Ann_0$ be a point such that $f(\tilde{x})$ specializes to $x$. For any fixed $n$, the map $\Ann_n \rightarrow \AgK^{\an}$ arises from a map of admissible formal schemes $g_n: \fA_n \rightarrow \AgK$, where $\fA_n$ is as in  \autoref{prop:specialfiber} -- in particular, the reduced, geometric special fiber $C_n \subset \fA_n$ is connected, its intersection graph is a tree, and each irreducible component is isomorphic to either $\P^1$ or $\A^1$. 

We now consider the identity map on $\abvar_{x}$ as a quasi-isogeny. By  \autoref{prop:isogenyovertreeofcurves}, this quasi-isogeny extends uniquely to a polarized quasi-isogeny $\rho:  \abvar_{x}\times C_n \rightarrow g_n^{-1}(\abvar)_{C_n}$. Further, as quasi-isogenies uniquely deform over nilpotent extensions, we obtain a polarized quasi-isogeny (again uniquely determined!) $\rho:  \abvar_{x}\times \fA_{\F_p} \rightarrow g_n^{-1}(\abvar)_{\fA_{\F_p}} $. This data uniquely defines a lift $\tilde{f}_{n,\Breve{K}}: \Ann_0\rightarrow \RZ^{\rig}_x$. Further, note that the image of $\tilde{x}$ was fixed once we defined $\rho$ to extend the identity map on $\abvar_{x}$. 

We now let $\mathfrak{T}\subset \AgK^{\an}$ denote the tube over the locus of points in $\AgK(\Fpbar)$ that are $p$-power quasi-isogenous to $x \in \AgK(\Fpbar)$, and $\mathfrak{T}^\circ$ the connected component of $\mathfrak{T}$ that meets the image of $f$. 
Let $(\RZ_x^{\circ})^{\rig}$ denote a connected component of $\RZ_x$ such that the natural map $j : (\RZ_x^{\circ})^{\rig} \rightarrow \mathfrak{T}^\circ$ is a topological covering.  
As $j: \RZ_{x}^{\rig}\rightarrow \mathfrak{T}$ is a topological cover, any lift $\tilde{f}_{n,\Breve{K}}$ of $f_{n,\Breve{K}}$ is uniquely determined once the image of $\tilde{x}$ is fixed. It therefore follows that $\tilde{f}_{n,\Breve{K}}$ is compatible as we range over all positive integers $n$. The result for $\Dstar$ follows. The result for $\sD$ follows immediately from the result for $\Dstar$ along with the fact that $j$ is a topological cover.

We are now ready to treat the case of $\Dab$. As all the spaces we consider are taut, we can and will work in the category of Berkovich analytic spaces. We first claim that the image of $f$ is contained in a connected component $\mathfrak{T}^\circ$ of the image $\mathfrak{T}$ of $j : \RZ_x^{\rig} \rightarrow \AgK^\an$. We note that $\mathfrak{T}^\circ$ is a rational open subspace of $\AgK^\an$ and that $j : \RZ(x)^{\rig} \rightarrow \mathfrak{T}^\circ$ is a topological cover on the level of the corresponding Berkovich analytic spaces \cite[Ch II Theorem 7.2.4]{Andre}.
To prove the claim, we assume without loss of generality that $b=0$. It then follows from \autoref{prop:isogenyovertreeofcurves} that $f(\Db)$ is contained in $\mathfrak{T}^\circ$, thus proving the claim. The skeleton of the Berkovich space associated to $\Dab$ is contractible \cite{berkovich-local-contractibility}, and therefore the Berkovich space underlying $\Dab$ is also contractible. Since the map $j$ is a topological cover of Berkovich analytic spaces, $f$ lifts to $\RZ_x^{\rig},$ as desired. 
\end{proof}

\section{Borel extension for Rapoport-Zink spaces and Shimura varieties}\label{sec:extension for RZ}

The purpose of this section is to prove \autoref{main} and \autoref{introRZextension}, i.e. Rapoport-Zink spaces and Shimura varieties of abelian type satisfy the Borel extension property. We will first prove a very local extension result for Rapoport-Zink spaces (\autoref{extendRZgeneral} below). 
We will then deduce \autoref{main} and \autoref{introRZextension} at the end of this section.


\begin{theorem}\label{extendRZgeneral}
    Let $X$ be a smooth rigid-analytic variety over $F$, and let $f: \Dstar \times X \rightarrow \RZ^{\rig}$ be a rigid-analytic map defined over $F$. Then there exists a rational open neighbourhood $U_z \subset \sD \times X$ and an analytic extension of $f$ to $U_z$ for every classical point $z\in (\sD \times X) \setminus (\Dstar \times X).$
\end{theorem}

\begin{remark}
     Our proof goes through for general Rapoport-Zink spaces associated to (minuscule) local Shimura data.   
\end{remark}

\subsection{Extending the crystalline local system}\label{sec:extendinglocalsystemDab}
In this section, we will work in the setting of a map $f: \Dab \rightarrow \RZ^{\rig}_{\mathbb{X}}$. For ease of notation, we will let $\RZ$ denote the generic fiber of $\RZ^{\rig}_{\mathbb{X}}$. Recall that $\RZ$ is equipped with $ \Vdr$ and $\Tate$.  The first step of \autoref{extendRZgeneral} is the following result.

\begin{theorem}\label{localsystemextendscrystalline}
The local system $f^*\Tate$ on $\Dab$ extends to $\sD^{a+b}$. Further, after possibly replacing $F$ by a finite extension, this extension has crystalline fibers at all classical points of $\sD^{a+b}$. 
\end{theorem}

In order to prove \autoref{localsystemextendscrystalline}, we will first extend the vector bundle with connection:

\begin{lemma}\label{vectorbundleconnectionextends}
Let notation be as in \autoref{localsystemextendscrystalline}. Then $f^{*}\Vdr$ extends to a vector bundle with (non-singular) connection on $\mathsf{D}^{1+b}$. 
\end{lemma}
\begin{proof}
 Recall that $\Vdr/\RZ$ admits a canonical trivialization \autoref{thm:triviality-of-vbconnection-RZ}. We pull this trivialization back to $\Dab$, and use it to extend $f^*\Vdr$ to a vector bundle with (non-singular) connection on $\sD^{1+b}$.  
\end{proof}

To finish the proof of \autoref{localsystemextendscrystalline}, we will make use of the $p$-adic Riemann-Hilbert correspondence as developed in \cite{liu-zhu-rigidity} and \cite{DLLZ} to prove that the local system extends to $\sD^{1+b}$ as a de Rham local system, and then use work of Shimizu \cite{koji} to deduce crystallinity at the puncture. For a smooth rigid analytic variety $X$ over $K$ equipped with a normal crossing divisor $Z$, let $D_{\mathrm{dR},\log}$ denote the logarithmic arithmetic Riemann-Hilbert functor as in \cite{DLLZ}. It sends a de Rham $\mathbb Z_p$-local system $L$ on $U=X-Z$ to a vector bundle $D_{\mathrm{dR},\log}(L)$ on $X$ equipped with a logarithmic connection $\nabla^{\log}$ (along $Z$), and a decreasing filtration of $D_{\mathrm{dR},\log}(L)$ by subbundles satisfying Griffiths transversality. The restriction of $D_{\mathrm{dR},\log}(L)$ to $U$ is $D_{\mathrm{dR}}(L)$ as from \cite{liu-zhu-rigidity}.

We need the following property of $D_{\mathrm{dR},\log}$.
\begin{proposition} 
  Assume that $Z$ is smooth.
  The comparison isomorphism of \cite[Theorem 3.7.11]{DLLZ} is compatible with the quasi-unipotent monodromy on $R\Psi^{qu}(L)$ and the residue $\mathrm{Res}_Z(\nabla^{\log})$, via logarithm.
\end{proposition}
\begin{proof}
We inspect the proof of \cite[Theorem 3.7.5]{DLLZ} and follow notations of \emph{loc. cit.} So $Z$ will be denoted by a smooth divisor $D\subset X$ defined by $T_1=0$, where $T_1$ is part of the toric chart, and $L$ will be denoted by $\mathbb{L}$, which is an \'etale $\mathbb Z_p$-local system. 

We first assume that $\mathbb{L}$ has unipotent monodromy around $D$.
By \cite[Proposition 6.4.4]{DLLZ2}, we see that under the isomorphism (3.7.9) of \cite[Theorem 3.7.5]{DLLZ}, the automorphism of $\mathcal R\mathcal H(R\Psi^{qu}(L))$ induced by the unipotent monodromy on $R\Psi^{qu}(L)$ corresponds to the action of the generator $\gamma_1$ of the log fundamental group of $D^\partial$ on $\widehat{\mathbb L}^{\partial}$.  On the other hand, the map (3.7.8) of  \emph{loc. cit.} identifies the right hand side of (3.7.9) with the restriction of $\mathcal R\mathcal H (\mathbb L)$ to $D$. In addition, by \cite[Lemma 3.4.7]{DLLZ}, the residue of $\mathcal R\mathcal H (\mathbb L)|_D$ is $t^{-1}\log(\gamma_1)$. This gives the compatibility of monodromy and residue for $\mathcal R\mathcal H$. In the case $\mathbb L$ is de Rham, one takes Galois invariants to get the compatibility of $D_{\mathrm{dR},\log}$.
\end{proof}

\begin{lemma}\label{lem: extension away codim two}
    Suppose $Z\subset X$ is a normal crossing divisor, and we endow $X$ with the fs log structure coming from $Z$ as in \cite[Example 2.3.17]{DLLZ2}. Let $\mathring{X}$ be the complement of the union of any intersection of two components of $Z$. Let $L$ be a Kummer \'etale local system on $X_{\mathrm{ket}}$, which is an \'etale local system on $\mathring{X}$. Then $L$ is an \'etale local system on $X$.
\end{lemma}
\begin{proof}
The follows from rigid Abhyankar’s lemma. In particular see the last statement of \cite[Lemma 4.2.2]{DLLZ2}.
\end{proof}

\begin{theorem}\label{localsystemextendsifconnectiondoes}
Let $\mathsf{L}/\Dab$ denote a de Rham local system, and let $(\mathcal{V},\nabla)$ denote the vector bundle with connection associated to it via the $p$-adic Riemann-Hilbert correspondence. Suppose that $(\mathcal{V},\nabla)$ extends to a vector bundle with connection on $\sD^{1+b}$. Then, $L$ also extends to a local system on $\sD^{1+b}$.
\end{theorem}
\begin{proof}
We equip $\sD^{1+b}$ with the standard log structure associated to the divisor $Z=\sD^{1+b}-\Dab$. (See \cite{DLLZ2}.) Let $\Omega^{\log}$ denote the corresponding sheaf of log differentials.

Let $D_{\mathrm{dR},\log}$ denote the logarithmic arithmetic Riemann-Hilbert functor as in \cite{DLLZ}. Then $D_{\mathrm{dR},\log}(L)$ consists of a vector bundle $\widetilde{\mathcal{V}}$ on $\sD^{a+b}$ equipped with a log connection 
$
\nabla^{\log}\colon\widetilde{\mathcal{V}}\to \widetilde{\mathcal{V}}\otimes\Omega^{\log}$. In addition, the residue $\mathrm{res}(\nabla^{\log}): \widetilde{\mathcal{V}}|_{Z}\to \widetilde{\mathcal{V}}|_{Z}$ has eigenvalues in $[0,1)\cap \mathbb Q$ along $Z$. In addition $(\widetilde{\mathcal{V}},\nabla^{\log})|_{\Dab}\cong (\mathcal{V},\nabla)$. In particular, $(\mathcal{V},\nabla)$ has a structure as a meromorphic connection (meaning there are no essential singularities) and $(\widetilde{\mathcal{V}},\nabla^{\log})$ is a canonical extension of $(\mathcal{V},\nabla)$. On the other hand by our assumption, $(\mathcal{V},\nabla)$ extends to a vector bundle with a connection on $\sD^{1+b}$ (which automatically is a canonical extension of $(\mathcal{V},\nabla)$). By the uniqueness of the canonical extension, we see that $\nabla^{\log}$ has no log pole, or equivalently the residue of $\nabla^{\log}$ is zero along any $Z$.

Now we use the above proposition. Together with \cite[Proposition 6.4.4]{DLLZ2}, we see that the log fundamental group acts trivially on $\overline{\mathsf{L}}_{Z^\partial(\xi)}$, where $\xi$ is a log geometric point of $Z^{\partial}$ (this is the special case of the space $U_J^{\partial}(\xi)$ defined in \emph{loc. cit.} for $J$ a singleton),  where $\overline{\mathsf{L}}=j_{\mathrm{ket},*}L$ is the kummer \'etale local system and where $j_{\mathrm{ket}}\colon \Dab_{\mathrm{et}}=\Dab_{\mathrm{ket}}\to (\sD^{1+b})_{\mathrm{ket}}$ is the map of kummer \'etale sites. Therefore, the pushforward of $\overline{\mathsf{L}}$ along $(\sD^{1+b})_{\mathrm{ket}}\to (\sD^{1+b})_{\mathrm{et}}$ gives the desired extension of $\mathsf{L}$ (using \cite[4.4.24-4.4.27]{DLLZ2} and \autoref{lem: extension away codim two}). 
\end{proof}

The following result follows directly from Shimizu \cite[Theorem 7.4]{koji}.  
\begin{theorem}[Shimizu]\label{Shimizu}
Let $\mathsf{L} / \sD^{1+b}$ denote a \emph{horizontally de Rham} local system which has crystalline fibers over $\Dab$. Then, after possibly replacing $F$ by a finite extension, $\mathsf{L}$ has potentially crystalline fibers at all classical points of $\sD^{1+b}$. 

\end{theorem}

We recall from \emph{loc. cit.} that a local system is called horizontally de Rham if it is de Rham and the vector bundle with connection arising from the $p$-adic Riemann-Hilbert correspondence admits a full set of solutions. We now finish the proof of \autoref{localsystemextendscrystalline}. 
\begin{proof}[Proof of \autoref{localsystemextendscrystalline}]
    We have already seen that $f^*(\Tate)$ and $f^*(\Vdr)$ extend to $\sD^{1+b}$ in a way compatible with the $p$-adic Riemann Hilbert correspondence. Further, we used the canonical trivialization on $f^*(\Vdr)$ (induced by the canonical trivialization of $\Vdr/\RZ$) to extend it to $\sD^{1+b}$, and therefore the extension of $f^*(\Tate)$ to $\sD^{1+b}$ satisfies the hypotheses of \autoref{Shimizu}. The result follows.
\end{proof}

\subsection{Period maps and the proof of \autoref{extendRZgeneral}}
We first need the following results:

\begin{lemma}\label{lem:integrally-closed}

Let $U = \Spa(A,A^+)$ be a normal irreducible affinoid rigid-analytic variety. Let $\mathcal{O}_U$ denote the sheaf of rigid-analytic functions on $U$. Suppose $Y\subseteq U$ is a closed rigid-analytic subvariety that has everywhere positive codimension in $U$.

Then $A$ is integrally closed in $\mathcal{O}_U(U\setminus Y)$
\end{lemma}
\begin{proof}
Let $f \in \mathcal{O}_U(U\setminus Y)$ be integral over $A$. Then in particular $f : U\setminus Y \rightarrow F$ is rigid-subanalytic and therefore by \cite[Theorem 2.22]{p-adic-definable-chow}, extends to a meromorphic function $f \in \mathrm{Frac}(A).$ Since $A$ is integrally closed in its fraction field $\mathrm{Frac}(A)$, we get that $f \in A$. 
\end{proof}

\begin{corollary}\label{cor:finite-extension}

Let $Y \subset U$ be as in \autoref{lem:integrally-closed}. Let $h: Z\rightarrow W$ a finite morphism of rigid-analytic varieties, and let $f : (U\setminus Y) \rightarrow Z$ be a rigid-analytic map such that the composite map $g : (U\setminus Y) \xrightarrow{f} Z \xrightarrow{h} W$ extends to a rigid-analytic map $\widetilde{g} : U \rightarrow W$. Then $f$ extends to a rigid-analytic map $\widetilde{f} : U \rightarrow Z$ such that $\widetilde{g} = h\circ \widetilde{f}$. 
\end{corollary}
\begin{proof}
Since the question of whether $f$ extends across $Y$ is local in $U$ and in $W$, we may assume that $W$, and hence also $Z$, is affinoid. Let us suppose that $Z = \Spa(C,C^+)$ and $W = \Spa(B,B^+)$ and let us denote by $h^\sharp : B \rightarrow C$ the finite morphism of affinoid $F$-algebras induced by $h : Z \rightarrow W$. The map $f$ corresponds to a map $f^\sharp : C \rightarrow \mathcal{O}_U(U \setminus Y)$ of $F$-algebras, such that $f^\sharp(h^\sharp(B)) \subseteq A \subseteq \mathcal{O}_U(U\setminus Y).$ Since $C$ is integral over $h^\sharp(B)$, we obtain that $f^\sharp(C)$ is integral over $A$, and thus by \autoref{lem:integrally-closed}, we obtain that $f^\sharp(C) \subseteq A$. Therefore, $f$ extends to a rigid-analytic map $\widetilde{f} : U \rightarrow Z$.
\end{proof}

\begin{corollary}\label{lem:etale-extension}
Let $Y\subset U$ be as in \autoref{lem:integrally-closed}. Let $h : Z \rightarrow W$ be an \'etale morphism of \emph{affinoid} rigid-analytic spaces over $F$ and let $f : (U\setminus Y) \rightarrow Z$ be a rigid-analytic map such that the composite map $g : (U\setminus Y) \xrightarrow{f} Z \xrightarrow{h} W$ extends to a rigid-analytic map $\widetilde{g} : U \rightarrow W$. Then $f$ extends to a rigid-analytic map $\widetilde{f} : U \rightarrow Z$ such that $\widetilde{g} = h\circ \widetilde{f}$.

\end{corollary}
\begin{proof}
By passing to a further covering of $W$ by rational open subdomains, we may assume by \cite[Proposition 8.1.2]{fresnel-vanderput}, that $h$ factors as a composite $h : Z \xhookrightarrow{i} Z' \xrightarrow{h'} W$ where $i$ is an open immersion, and $h'$ is finite-\'etale. By \autoref{cor:finite-extension}, the composite $f' : (U\setminus Y) \xrightarrow{f} Z \xhookrightarrow{i} Z'$ extends to a rigid-analytic map $\widetilde{f'} : U \rightarrow Z'$ such that $\widetilde{g} = h' \circ \widetilde{f'}$. Furthermore, since $i(Z)$ being an affinoid subdomain of the affinoid space $Z'$ is closed (and open) in $Z$ in the naive metric topology and since $U\setminus Y$ is dense in $U$ in the naive metric topology, we get that $\widetilde{f'}(U) \subseteq i(Z)$. Now since $i$ is an open immersion, we get that $f$ also extends to a rigid-analytic map $\widetilde{f}:U \rightarrow Z$ such that $\widetilde{g} = h\circ \widetilde{f}$.
\end{proof}

\begin{lemma} \label{lem:etale-bijection-lifting}
  
  Let $Y\subset U$ be as in \autoref{lem:integrally-closed}. Let $h : Z \rightarrow W$ be an \'etale morphism of separated rigid-analytic varieties (not necessarily affinoid) over $F$ that is \emph{bijective} at the level of classical points. Let $f : (U\setminus Y) \rightarrow Z$ be a rigid-analytic map such that the composite map $g : (U\setminus Y) \xrightarrow{f} Z \xrightarrow{h} W$ extends to a rigid-analytic map $\widetilde{g} : U \rightarrow W$. Then for every classical point $z \in U$ there exists a rational open subset $U_z \subseteq U$, and an extension of $f\vert_{U_z\setminus Y} : U_z \setminus Y \rightarrow Z$ to an analytic map $\widetilde{f}_z : U_z \rightarrow Z$ such that $\widetilde{g}\vert_V = h\circ \widetilde{f}_z$. For any such open subspace $U_z$ the extension $\widetilde{f}_z$ is unique.
\end{lemma}
\begin{proof}
    Let $W = \cup_j W_j$ be an open cover of $W$ by rational subdomains. For each $j$, pick an open cover of $U_j := \widetilde{g}^{-1}(W_j)$ by connected rational affinoids, say $U_j = \cup_k U_{jk}$. We may replace $U$ by $U_{jk}$, $Y$ by $Y\cap U_{jk}$, $W$ by the affinoid $W_{j}$, and $Z$ by $h^{-1}(W_j)$, and thereby assume without loss of generality that $W$ is also affinoid. 

    Let $Z = \cup_j Z_j$ be an open cover of $Z$ by rational subdomains. The map $h|_{Z_j} : Z_j \rightarrow W$ is an \'etale map of affinoid spaces, and therefore $h(Z_j)$ is a finite union of rational subdomains of $W$ \cite[Proposition 8.1.2 (3)]{fresnel-vanderput}. 
    For each $j$, pick an open cover of $U_j := \widetilde{g}^{-1}(W_j)$, say $U_j = \cup_k U_{jk}$ by \emph{connected} rational affinoids $U_{jk}$. The affinoids $U_{jk}$ being connected and normal are in particular irreducible \cite[Lemma 2.1.4 and Lemma 2.2.3]{conrad-irreducible-comp}.
    We may thus apply \autoref{lem:etale-extension}, to get for each $j, k$ a unique map $\widetilde{f}_{jk} : U_{jk} \rightarrow Z_j$ making the following diagram commute:
\[
    \begin{tikzcd}
        U_{jk}\setminus Y \arrow[r, "f\vert_{U_{jk}\setminus Y}"] \arrow[d, hook] & Z_j \arrow[d, "h\vert_{Z_j}"] \\
        U_{jk} \arrow[ur, dotted, "\widetilde{f}_{jk}" description] \arrow[r, "\widetilde{g}"] & W 
    \end{tikzcd}
\]
 The uniqueness of the lifts $\widetilde{f}_{jk}$ implies that the $\widetilde{f}_{jk}$ agree on pairwise intersections as rigid-analytic maps into $Z$, and thus the $\widetilde{f}_{jk}$ glue to a lift $\widetilde{f} : \cup_j U_j \rightarrow Z$. We now note that $\cup_j U_j$ is an open subspace of $U$ that contains all the classical points of $U$. This completes the proof.
\end{proof}

\begin{lemma} \label{lem:almost-etale-lifting}
  Let $U = \Sp(A)$ be a normal, irreducible affinoid rigid-analytic variety over $F$. Suppose that $Y \subseteq U$ is a closed rigid-analytic subvariety that is everywhere of positive codimension in $U$. Let $h : Z \rightarrow W$ be an almost  \'etale covering map of rigid-analytic spaces (not necessarily affinoid) over $F$ and let $f : (U\setminus Y) \rightarrow Z$ be a rigid-analytic map such that the composite map $g : (U\setminus Y) \xrightarrow{f} Z \xrightarrow{h} W$ extends to a rigid-analytic map $\widetilde{g} : U \rightarrow W$. Then $f$ extends to a unique rigid-analytic map $\widetilde{f} : U \rightarrow Z$ such that $\widetilde{g} = h\circ \widetilde{f}$.   
\end{lemma}
\begin{proof}
    Let $W = \cup_j W_j$ be an open affinoid cover of $W$ such that $Z_j := h^{-1}(W_j)$ has an open cover by a disjoint union, $Z_j = \coprod_k Z_{jk}$, where the $Z_{jk}$ are affinoid spaces such that $h \vert_{Z_{jk}} : Z_{jk} \rightarrow  W_j$ is a finite \'etale map. 

    For each $j$, pick an open cover of $U_j := \widetilde{g}^{-1}(W_j)$ by connected (hence irreducible) affinoids, say $U_j = \cup_l U_{jl}$. Note that $U_{jl}\setminus Y$ is also irreducible. The irreducible space $U_{jl}\setminus Y$ has an open cover by the disjoint union of the rational open subsets $\coprod_{k} f^{-1}(Z_{jk})\cap (U_{jl}\setminus Y)$. Therefore, for each $j, l$, there is a single $k = k(j,l)$ such that $f(U_{jl}\setminus Y) \subseteq Z_{jk}$. Applying \autoref{lem:etale-extension} we get for each $j, l$, a lift $\widetilde{f}_{jl} : U_{jl} \rightarrow Z_{jk}$ making the following diagram commute:
    \[
    \begin{tikzcd}
        U_{jl}\setminus Y \arrow[r, "f\vert_{U_{jl}\setminus Y}"] \arrow[d, hook] & Z_{jk} \arrow[d, "h\vert_{Z_{jk}}"] \\
        U_{jl} \arrow[ur, dotted, "\widetilde{f}_{jl}" description] \arrow[r, "\widetilde{g}"] & W_j 
    \end{tikzcd}
\]
The lifts $\widetilde{f}_{jl}$ agree on pairwise intersections and hence glue to provide a lift $\widetilde{f} : U \rightarrow Z$, thereby proving the lemma.
\end{proof}

We are now ready to prove the main theorem of this section. 
\begin{proof}[Proof of \autoref{extendRZgeneral}]
    We may assume that $X$ is smooth affinoid. 
    Recall that we have the maps $f: \Dstar \times X \rightarrow \RZ^{\rig}$, and the \'etale period map $\check{\pi}: \RZ^{\rig} \rightarrow \cF^\an$. 
    The image of this period map is contained in the weakly admissible open subset $\cF^{wa} \subseteq\cF^{an}$ consisting of the weakly admissible filtrations.
   We let $U \subset \cF^{wa}$ denote the image of this map (see \cite{deJong-etale}).
    The morphism $\check{\pi}$ factors through $U$, $\check{\pi} : \RZ^{\rig} \rightarrow U \rightarrow \cF^{wa}$, wherein due to \cite[Theorem 8.4 (a) \& Prop. 6.3]{hartl}, we have that the map $U \rightarrow \cF^{wa}$ is an \'etale morphism of rigid analytic spaces that is \emph{bijective} at the level of classical points (see \cite{deJong-etale}). Furthermore, again using \cite{deJong-etale} the map $\RZ^{\rig} \rightarrow U$ is an almost \'etale covering map.

    \textbf{Step 1:} \emph{The composite map $\varphi := \check{\pi} \circ f : \Dstar \times X \rightarrow \cF^{wa}$ extends analytically to a map $\widetilde{\varphi} : \sD \times X \rightarrow \cF^{wa}$.}
    
    By \autoref{hyperlocalextensionforX}, to prove Step 1, it suffices to prove that every classical point $z\in \sD \times X \setminus \Dstar \times X$ has a neighbourhood $z \in U_z \subset \sD\times X$, such that $\varphi$ extends analytically to a map $U_z \rightarrow \cF^{wa}$. Therefore, we can reduce from $\Dstar\times X \subset \sD\times X$ to the case of $\Dab \subset \sD^{1+b}$. In particular, we may use the results proved in Section \ref{sec:extendinglocalsystemDab}.

    Let $\hat{\Tate}$ denote the extension of $f^*(\Tate)$ to $\Dab$. 
    Applying the $p$-adic Riemann-Hilbert correspondence, we obtain that $D_{dR}(\hat{\Tate})$ is an analytic vector bundle with integrable connection on $(\sD)^{a+b}$ endowed with a filtration that extends $\Vdr$ along with its Hodge filtration. Recall that the extension of $\Vdr$ to $\sD^{a+b}$ is the trivial vector bundle with connection. The extension of the filtration across the punctures of $\sD^{a+b}$ gives us the extension\footnote{We remark that \cite{DLLZ} gives us an extension of the map even when $L$ doesn't extend to a de Rham local system on all of $\sD$, because \cite{DLLZ} gives a filtered vector bundle with log-connection on all of $\sD$, and the only data needed to extend the map to all of $\sD$ is the data of the filtration at $0$.} of $\varphi : \Dab \rightarrow \cF^\an$ to $\tilde{\varphi} : \sD^{1+b} \rightarrow \cF^\an$. After possibly replacing $F$ by a finite extension, the crystallinity of the fibers of $\hat{\Tate}$ at all points of $\sD^{1+b}$ (by \autoref{localsystemextendscrystalline}) implies that $\tilde{\varphi}(\sD^{1+b})$ is contained in the weakly admissible locus $\cF^{wa}$ of $\cF$. This completes Step 1.
    
    \textbf{Step 2:} \emph{Completion of the proof.}

    By Step 1 and \autoref{lem:etale-bijection-lifting}, we get that for every classical point $z \in \sD \times X$, there is a connected affinoid open subdomain $U_z \subseteq \sD \times X$, such that the composite map $U_z\cap (\Dstar \times X) \rightarrow \RZ^{\rig} \rightarrow U$, extends analytically to a map $U_z \rightarrow U$. We now conclude the proof of the theorem by applying \autoref{lem:almost-etale-lifting} to each $U_z$, and recalling the fact that $\RZ^{\rig} \rightarrow U$ is an almost \'etale covering map. 

\end{proof}

We now prove the extension result for Shimura varieties of abelian type and Rapoport-Zink spaces. 

\begin{proof}[Proof of \autoref{main}]
The proof is just an exercise in assembling together the various results proved earlier in the paper. First, by \autoref{prop:reducetoAgpar} we may assume that $S(G,X)_{\sK} = \AgK$ where with full level structure at $\ell$ and parahoric level structure at $p$, and that $X=\sD^b$. By \autoref{thm:reducetocaseofgoodred}, we may assume that $f(\Dab) \subset \AgK^{\good}$. We then have (by \autoref{thm:maptoRZ}) that $f$ lifts to a map $\tilde{f}: \Dab \rightarrow \RZ^{\rig}$ for an appropriate Rapoport-Zink space $\RZ^{\rig}$. By \autoref{extendRZgeneral}, we have that there exists a neighbourhood $U_z \subset \Dab$ and an extension of $\tilde{f}$ to $U_z$ for every $z\in \sD^{1+b} \setminus \Dab$, whence an extension of $f$ to $U_z$. The theorem now follows from \autoref{hyperlocalextensionforX}.
\end{proof}

\begin{proof}[Proof of \autoref{introRZextension}]
   Recall that the map from Rapoport-Zink spaces onto their image in $\AgK$ is a topological cover. Therefore, the result follows from the extension theorem for $\AgK$ and \autoref{hyperlocalimpliesgeneraltopologicalcover}. 
\end{proof}

\begin{remark}
    Here we indicate another possible approach to \autoref{introRZextension}, which does not rely on \autoref{main}. Namely, $U\subset \cF^{wa}$ in the above proof is an open subset (in the sense of adic spaces) on which there exists a sympletic $\mathbb{Q}_p$-local system. Then $\RZ^{\rig}$ admits a moduli interpretation as the space of $\mathbb{Z}_p$-lattices certain type in this local system. Since on $\sD^{1+b}$ we do have such a $\mathbb{Z}_p$-local system, we obtain a map $\sD^{1+b}\to \RZ^{\rig}$ as desired.
\end{remark}

\section{Proof of the algebraicity theorems}\label{sec:last}

In this section, we will prove the algebraicity results and deduce the corollaries to the main results.  First, note that \autoref{algebraicityofabeloids} follows directly from \autoref{thm:borelalg} and the fact that $\Ag^{\an}$ is the moduli space of polarized abeloids. Similarly, \autoref{extensionforabeloids} follows also follows from this fact and \autoref{toroidalextension}.  \autoref{cor:Nadel} follows directly from \autoref{thm:borelalg} and \cite[Theorem 0.1]{Nadel}. Therefore, it suffices to prove \autoref{thm:borelalg}, \autoref{toroidalextension} and \autoref{nomapsfromgroups}. For the convenience of the reader, we recall the statement of \autoref{thm:borelalg}.
\begin{corollary*}
    Let $S(G,\mathcal{H})_{\sK}$ denote a Shimura variety of abelian type with torsion-free level structure, and $V$ be an algebraic variety over $F$. Then, every rigid-analytic map $f  : V^{\an} \rightarrow S(G,\mathcal{H})_{\sK}^\an$ over $F$ is algebraic.
\end{corollary*}
\begin{proof}
    It suffices to prove that the restriction of $f$ to a Zariski dense open subscheme of $V$ is algebraic. Replacing $V$ by its regular locus, we may thus assume that $V$ is a smooth affine variety over $F$.

    By Hironaka's resolution of singularities \cite{hironaka-resolutionI}, we may find a smooth, proper variety $\hat{V}$ over $F$ and an open embedding $V \subseteq \hat{V},$ such that $\partial {V} := \hat{V}\setminus V$ is a divisor with simple normal crossings inside $\hat{V}$. By GAGA, it suffices to extend $f$ analytically to a map $\hat{V}^\an \rightarrow (S(G,\mathcal{H})_\sK^\mathrm{BB})^\an$.
    Denote the distinct irreducible components of $\partial {V}$ by $E_1,\ldots,E_m$. We note that each $E_i$ (given its reduced closed subscheme structure) is smooth over $F$. 
    For each $i \in \{1,\ldots,m\}$, let $E_i^\circ := E_i \setminus \cup_{r \neq i} E_r$, and let $V^\circ := V\bigcup \cup_{1\leq i \leq m} E_i^\circ \subseteq \hat{V}.$ By Kiehl's theorem on tubular neighbourhoods \cite[Theorem 1.18]{kiehl-derham}, we have an affinoid open cover of $(V^\circ)^\an$, $ (V^\circ)^\an = \cup_{j \in J} W_j, $ such that for each $j \in J$ such that $W_j \cap \partial {X} \neq \emptyset$, we have an isomorphism of open immersions 
        \[
    \begin{tikzcd}
        (W_j \cap V^\an) \arrow[r, hook] \arrow[d, "\simeq"] & W_j \arrow[d, "\simeq"] \\
        (W_j \cap \partial {V}) \times \Dstar  \arrow[r, hook] & (W_j \cap \partial {V}) \times \sD 
    \end{tikzcd}
\]
    Applying \autoref{main}, we see that $f\vert_{W_j\cap V^\an}$ admits a (unique) analytic extension to an analytic map $W_j \rightarrow (S(G,\mathcal{H})_\sK^\mathrm{BB})^\an$, for each $j \in J$, and therefore $f$ extends to an analytic map $f^\circ : (V^\circ)^\an \rightarrow (S(G,\mathcal{H})_\sK^\mathrm{BB})^\an.$

    For each $i < j < k$ in $\{1,\ldots,m\}$, set $E_{ij} := E_i \cap E_j$, and $E_{ijk} := E_i\cap E_j \cap E_k$.  
    Let \[E_{ij}^\circ := E_{ij} \setminus (\cup_{r<s<t} E_{rst})\]
    and let  $V^{\circ\circ} := V^\circ \bigcup \cup_{1\leq i < j \leq m} E_{ij}^\circ.$ Set $\partial V^{\circ} := \hat{V}\setminus V^{\circ}.$ Note that for each $i < j$, $E_{ij}^\circ$ is smooth over $F$.
    Again, applying Kiehl's theorem on tubular neighbourhoods we see that there is an open affinoid covering, $({V^{\circ\circ}})^\an = \cup_{j \in J} V_j,$ such that for each $j \in J$ such that $V_j \cap \partial V^{\circ} \neq \emptyset$, we have an isomorphism of open immersions: 

    \[
    \begin{tikzcd}
        (V_j \cap (V^\circ)^\an) \arrow[r, hook] \arrow[d, "\simeq"] & V_j \arrow[d, "\simeq"] \\
        (V_j \cap \partial {V^\circ}) \times \sD^2\setminus\{0\}  \arrow[r, hook] & (V_j \cap \partial {V^\circ}) \times \sD^2 
    \end{tikzcd}
\]

Applying the extension result, \autoref{main} now implies that $f^\circ$ extends to a rigid analytic map $f^{\circ\circ} : (V^{\circ\circ})^\an \rightarrow (S(G,\mathcal{H})^\mathrm{BB}_\sK)^\an.$ 
We proceed in this manner, iteratively as above to extend $f$ analytically to $\hat{V}$, at each step using Kiehl's theorem.   
\end{proof}

We now also deduce our result, \autoref{toroidalextension}, on extending rigid-analytic families of abelian varieties and K3 surfaces over $\Dstar$. 

\begin{proof}[Proof of \autoref{toroidalextension}]
The proof is similar to the complex case, and uses the valuative criterion for properness.

\end{proof}

Finally, we now deduce the proof of \autoref{nomapsfromgroups}. 

\begin{proof}[Proof of \autoref{nomapsfromgroups}]
    By \autoref{thm:borelalg}, it suffices to show that every algebraic map $\mathbb{G} \rightarrow S(G,\cH)_{\sK}$ is constant, and therefore we may work over $\C$. This is presumably well known to experts, so we will be content with a brief sketch. Using the structure of alegbraic groups, it suffices to prove constancy in the cases of abelian varieties, tori, semisimple groups, and unipotent groups. Abelian varieties and tori are covered holomorphically by $\C^g$. As there are no non-constant holomorphic maps $\C\rightarrow S(G,\cH)_{\sK}^{\an}$, the result follows for abelian varieties and tori, and also for unipotent groups. Semisimple groups have finite fundamental groups, and therefore constancy follows from Deligne's ``theorem of the fixed part'' \cite[4.1.2]{Delignefixedpart}.    

 \end{proof}

\appendix

\section{Formal models of rigid annuli}

Throughout this appendix,  let $R$ denote a discrete valuation ring with fraction field $F$, and with residue field $k$. For any $R$-scheme $Y/R$,  let $Y_F$ and $Y_k$ denote the generic and special fibers of $Y \to \Spec R$ respectively.

The purpose of this appendix is to prove \autoref{prop:specialfiber} below, which provides a characterization of the irreducible components in the special fibers of admissible formal models of rigid-analytic annuli, along with information on how these components intersect with one another.  The statement of the proposition refers to the {\sl intersection graph} of a $\kbar$-curve $Z$ which is defined as follows. 

\begin{definition}
The {\sl intersection graph} of a $\kbar$-curve $Z$ is the graph whose vertices are those points of $Z$ lying on at least $2$ irreducible components, with one edge connecting distinct vertices $p$ and $q$ for each irreducible component of $Z$ containing both $p$ and $q$.
\end{definition}

\begin{lemma}
 \label{lemma:AAA}
 Suppose $X/R$ is an integral, regular $R$-scheme proper and flat over $\Spec R$, and suppose $X_F$ is a smooth genus $0$ curve with $H^0(X_F, \O_{X_{F}}) = F$. Then  the curve $(X_{\kbar})_{\red}$ is connected, its irreducible components are isomorphic to $\P^1_{\kbar}$, and its intersection graph is a tree.
\end{lemma}

\begin{proof}
 Let $C_1, \dots, C_n$ denote the irreducible components (with reduced scheme structure) of the Cartier divisor $X_k \subset X$, and for each $i =1,\dots, n$ let $m_i$ denote the multiplicity of $C_i$ in $X_k$.  As $X$ is regular, each $C_i$ is also a Cartier divisor in $X$. Setting $d = \gcd (m_1, \dots, m_n)$, we let $D \subset X$ denote the Cartier divisor $\sum_{i} (\frac{m_i}{d})C_i$. By \cite[\href{https://stacks.math.columbia.edu/tag/0C69}{Tag 0C69}]{stacks-project}, we get: 
 \begin{align*}
     1 = d[\kappa:k](1-g_D),
 \end{align*}
 where $\kappa = H^0(D, \O_D)$, and $g_D = \dim H^1(D, \O_D)$. Thus $g_D=0$.  As $\dim X_k = 1$, it follows that the  map $$H^1(D,\O_D) \to H^1((X_k)_{\red},\O_{(X_k)_{\red}})$$ induced by the natural quotient map $\O_D \to \O_{(X_k)_{\red}}$ is surjective, and hence $$ H^1\left((X_k)_{\red},\O_{(X_k)_{\red}}\right) = 0.$$ Flat base change then gives 
 \begin{align}
 \label{eq:H1zero}
     H^1\left((X_{\kbar})_{\red},\O_{(X_{\kbar})_{\red}}\right) = 0.
 \end{align}
 
 For the rest of the proof, let $\Gamma$ denote the reduced, proper $\kbar$-curve $\left(X_{\kbar} \right)_{\red}$.  By \cite[\href{https://stacks.math.columbia.edu/tag/0AY8}{Tag 0AY8}]{stacks-project} we see that $\Gamma$ is connected. If $\Gamma' \subset \Gamma$ is any sub-curve of $\Gamma$ (meaning a  $1$-dimensional closed subscheme), then the map $$H^1(\Gamma, \O_{\Gamma}) \to H^1(\Gamma', \O_{\Gamma'})$$ induced by the natural quotient map of sheaves $\O_{\Gamma} \to \O_{\Gamma'}$ is surjective (using the long exact sequence and that $\Gamma$ is $1$-dimensional) and therefore  \eqref{eq:H1zero} implies $H^1(\Gamma', \O_{\Gamma'})=0$. An easy argument using the finiteness of normalization shows that if $E$ is an integral and proper curve over an algebraically closed field, then the vanishing of $H^1(E,\O_{E})$ implies $E$ is isomorphic to $\P^1$.  Applying this fact to each irreducible component of $\Gamma$ it follows that every irreducible component of $\Gamma$ is isomorphic to $\P^1_{\kbar}$.

 Finally, for sake of contradiction, suppose $\Gamma_j$, $j \in \Z/m \Z$, $m \geq 2$ are irreducible components of $\Gamma$ such that there exist points $p_{j} \in \Gamma_j \cap \Gamma_{j+1}$ for all $j$ with the further property that $p_i \neq p_j$ whenever $i \neq j$. (This is negating the property that the intersection graph is a tree.) Let $$Z =  \bigcup_{j \in \Z/ m\Z} \Gamma_{j} \subset \Gamma,$$ and let $$Z' = \coprod_{j \in \Z/m\Z} \Gamma_j$$ denote the normalization of $Z$ with normalization map $\eta: Z' \to Z$.  From the previous paragraph, we know that $H^1(Z, \O_Z) = 0$.  Consider the short exact sequence of sheaves 
 \begin{align}
     0 \to \O_Z \to \eta_{*}\O_{Z'} \to \mathcal{C} \to 0
 \end{align}
 where $\mathcal{C}$ denotes the cokernel, supported on the set $S \subset \Gamma$ of intersections among the $\Gamma_{j}$'s, of the map $\O_Z \to \eta_{*}\O_{Z'}$. As $H^1(Z, \O_Z) = 0$, the long exact sequence implies that the induced map on global sections 
 \begin{align}
 \label{eq:surjection}
 H^0(Z, \eta_{*}\O_{Z'}) \to H^0(Z, \mathcal{C})
 \end{align} is surjective.  We will now contradict this surjectivity.  The domain vector space in \eqref{eq:surjection} is isomorphic to $H^{0}(Z', \O_{Z'})$ and hence is $m$-dimensional.  As $\mathcal{C}$ is supported on $S$, the latter vector space is at least $\# S$-dimensional. But $\# S \geq m$ because each $p_j \in S$.  Thus, the induced map 
 $$H^{0}(Z',\O_{Z'})/H^{0}(Z, \O_{Z}) \to H^0(Z, \mathcal{C})$$ cannot possibly be surjective for dimension reasons, contradicting surjectivity of \eqref{eq:surjection} and the lemma is proved.
\end{proof}

\begin{remark}
\label{remark:normal} Using Zariski's Main Theorem, we can relax the assumption ``$X$ is regular" to ``$X$ is normal" in \autoref{lemma:AAA}.  It is unlikely that the regularity assumption can be relaxed much further in light of the classical example of a family of twisted cubics in $\P^3$ specializing to an irreducible, planar, nodal curve having an embedded point at the node. 
\end{remark}

Let $A = R[x,y]/(xy-\pi)$ and set $Y := \Spec A$.  $Y$ is an integral, regular, $R$-scheme flat over $R$ with relative dimension $1$.  Next, suppose $I \subset A$ is an ideal containing some positive power of $\pi \in A$.  Therefore, the closed subscheme $V(I)$ is supported on the special fiber $Y_k = \Spec k[x,y]/(xy)$.  

\begin{lemma}
 \label{lemma:replaceI} Let $J = (f_1, \dots, f_m) \subset S$ be an ideal in an integral domain $S$ such that there exists an element $h \in S \setminus \{0\}$ and elements $g_1, \dots, g_m \in S$ satisfying $f_i = hg_{i}$ for all $i$.  Then if $J' = (g_1, \dots, g_m)$, then the schemes $\Bl_{J}S$ and $\Bl_{J'} S$ are canonically isomorphic as $S$-schemes.       
\end{lemma}

\begin{proof}
 The scheme $\Bl_{J}S = \Proj (S \oplus J \oplus J^2 \oplus \dots)$ has its usual affine open coverings \[U_{i} = \Spec S_i,\] where the ring $S_{i} := S\left[\frac{f_{1}}{f_{i}}, \dots, \frac{f_{m}}{f_{i}}\right]$ is the $S$-subalgebra of $\Frac (S)$ generated by the listed elements.  The open sets $U_i$ and $U_j$ are glued in the traditional way. Clearly, $S_i = S\left[\frac{g_{1}}{g_{i}}, \dots, \frac{g_{m}}{g_{i}}\right]$, and therefore the analogous standard affine opens $V_i$ of $\Bl_{J'}S = \Proj (S \oplus J' \oplus J'^2 \oplus \dots)$ are naturally identified with $U_i$, in a way which respects the gluing data.  The lemma follows.
\end{proof}

Suppose $f_1, \dots, f_{m} \in I$ generate the ideal $I \subset A$ (which contains some positive power of $\pi$).  If all $f_i$ are divisible by $x \in A$ or $y\in A$, then let $I'$ be the ideal obtained by dividing out first by the highest power of $x$  and then the highest power of $y$ dividing all $f_i$. 
\begin{lemma}
 \label{lemma:I'} The schemes $\Bl_{I}A$ and $\Bl_{I'}A$ are canonically isomorphic, and the closed subscheme $V(I') \subset Y$ is $0$-dimensional, supported on the special fiber $Y_{k}$.
\end{lemma}

\begin{proof}
The first assertion about the canonical isomorphism follows from \autoref{lemma:replaceI}. For the second assertion, we can and will replace $I'$ by its radical $\sqrt{I'}$.  Thus $V(I')$ is a subscheme of the special fiber $Y_k$.  Let $g_{1}, \dots, g_{m} \in I'$ be generators -- by assumption, the $g_i$ are not all divisible by $x$ or $y$. It follows that $V(I')$ is a proper (in the sense of sets) closed subscheme of $Y_k$, and hence is $0$-dimensional, as claimed in the lemma. 
\end{proof}

\begin{lemma}
 \label{lemma:AAimprovement} Suppose $I \subset A$ is an ideal containing some positive power of $\pi$.  There exists an integral, regular $R$-scheme $X/R$, flat over $\Spec R$, and a proper morphism $\rho: X \to \Bl_{I}A$ such that the composite map $X \to Y$ is an isomorphism over the open set $U = Y \setminus V(I)$.  Furthermore, the curve $(X_{\kbar})_{\red}$ is a tree of rational curves with two irreducible components isomorphic to $\A^1_{\kbar}$ and all others isomorphic to $\P^1_{\kbar}$.
\end{lemma}

\begin{proof}
We will need to compactify $Y$ to a proper $R$-scheme simply to be able to use \autoref{lemma:AAA}.  To that end, let \[\overline{Y} = \Proj R[X,Y,Z]/(XY-\pi Z^2).\]  The scheme $Y = \Spec R[x,y]/(xy-\pi)$ is naturally identified with the open subscheme of $\overline{Y}$ obtained by removing the closed subscheme defined by the ideal $(Z)$. (This closed subscheme is the union of the two $R$-points $[0:1:0]$ and $[1:0:0]$.)  The scheme $\overline{Y}$ is regular, and it is proper and flat over $R$ with generic fiber isomorphic to $\P^1_{F}$. (It has a $F$-rational point, for instance  $[0:1:0]$.)

By \autoref{lemma:I'}, we may and will assume that the closed subscheme of $Y$ determined by $I$ is a finite, $0$-dimensional scheme $W$ supported on the special fiber $Y_k$. Thus, the underlying set of $W$ remains a closed subset of $\overline{Y}$, and therefore $W$ is a closed subscheme of $\overline{Y}$ entirely contained in $Y \subset \overline{Y}$.  

Now, let $B = \Bl_{W} \overline{Y}$ -- by the previous paragraph, it follows that $\Bl_{I}A$ is an open subscheme of $B$, and that the natural morphism $B \to \overline{Y}$ is an isomorphism over the open subscheme $V = \overline{Y} \setminus W$.  We can and will consider $V$ as an open subset of $B$ as well.  As $\overline{Y}$ is regular, it follows that $V$ is also regular.

Next, by \cite[\href{https://stacks.math.columbia.edu/tag/0C2U}{Tag 0C2U}]{stacks-project}, the $2$-dimensional scheme $B$ admits a desingularization \[\rho: \overline{X} \to B,\] obtained by performing blow-ups and normalizations in such a way that $\rho: \overline{X} \to B$ is an isomorphism over $V \subset B$. ($B$ is regular at all points of $V$.)

The special fiber $\overline{X}_{k}$ has two distinguished irreducible components which we denote by $L$ and $M$ -- their generic points lie in $V$, and correspond to the generic points of the two $\P^1$'s $V(X,\pi)$ and $V(Y, \pi)$ comprising the special fiber $\overline{Y}_{k}$.  All other irreducible components of the curve  $\overline{X}_{k}$ are contracted to the $0$-dimensional scheme $W \subset \overline{Y}$. 

Finally, we set $X \subset \overline{X}$ to be the open set $\rho^{-1}(Y)$.  Observe that $L \cap X$ and $M \cap X$ are two copies of $\A^{1}_{k}$, as we have simply removed the points lying over $[0:1:0]$ and $[1:0:0]$ from what were originally $\P^1$'s. Now, the lemma follows by applying \autoref{lemma:AAA}, which says that all other irreducible components of the reduced geometric special fiber $(X_{\kbar})_{\red}$ not arising from $L$ or $M$ are isomorphic to $\P^{1}_{\kbar}$, and that the intersection graph is a tree.
\end{proof}

\begin{remark}
\label{anandremark}
From the proof of \autoref{lemma:AAimprovement} above, the two $\A^{1}$ components $L \cap X, M \cap X$ of $(X_{\kbar})_{\red}$ map properly and birationally onto the two components $V(\pi, x)$ and $V(\pi,y)$ under the blow down morphism $X \to \Bl_{I} A$. A proper, birational mapping between two smooth curves is an isomorphism, so in particular, in each $\A^{1}$ component of $(X_{\kbar})_{\red}$ there is a unique point mapping to the closed point $V(\pi,x,y) \in Y$.
\end{remark}

For the following proposition, we additionally assume that $R$ is a \emph{complete} discrete valuation ring.

\begin{proposition}\label{prop:specialfiber}
Let $\Ann := \Spa(F\langle x, y\rangle /( xy-\pi^n ), R\langle x, y\rangle /( xy-\pi^n ))$ be a rigid-analytic annulus over $F$, and let $\mathfrak{X}$ denote an admissible formal scheme over $R$ that is quasi-paracompact. Suppose there exists a rigid-analytic map $f: \Ann \rightarrow \mathfrak{X}^{\rig}$. Then, there exists an admissible formal $R$-scheme $\mathfrak{A}$ with
$\mathfrak{A}^{\rig} = \Ann$ and a map $g: \mathfrak{A}\rightarrow \mathfrak{X}$ of admissible formal schemes over $R$ with $g^{\rig} = f$, such that the reduced, geometric special fiber of $\mathfrak{A}$ is connected, and each of its irreducible components is isomorphic to either $\P^1_{\overline{k}}$ or $\A^1_{\overline{k}}$, and its intersection graph is a tree.
\end{proposition}

\begin{proof}
Let $\mathfrak{A}_0 := \Spf(R\langle x, y\rangle /( xy - \pi^n ))$, so that $\mathfrak{A}_0$ is an admissible formal $R$-scheme with $\mathfrak{A}_0^{\rig} \cong \Ann.$ 
Let us also set $A := R[x , y]/( xy - \pi^n )$, so that the $\pi$-adic completion $\hat{A}$ of $A$ is precisely $R\langle x,y\rangle /( xy - \pi^n)$. By \cite[\S 8.4, Lemma 4]{bosch}, we see that there is a coherent open ideal $\hat{J} \subseteq \hat{A}$, and a map of admissible formal $R$-schemes, $h : \Bl_{\hat{J}}\mathfrak{A}_0 \rightarrow \mathfrak{X}$, such that the following holds; if $\alpha : \Bl_{\hat{J}}\mathfrak{A}_0 \rightarrow \mathfrak{A}_0$ denotes the admissible formal blow-up of $\mathfrak{A}_0$ along the coherent open ideal $\hat{J}$, then $h^{\rig} = f \circ \alpha^{\rig},$ with $\alpha^{\rig}$ being an isomorphism.

The ideal $\hat{J}$, being an open ideal of $\hat{A}$, contains a power of $\pi$, and thus there exists an ideal $J$ of $A$ such that the extension of $J$ in $\hat{A}$, is $\hat{J}$. Furthermore, we recall that the formal blow-up $\alpha : \Bl_{\hat{J}}\mathfrak{A}_0 \rightarrow \mathfrak{A}_0$ is in fact the $\pi$-adic completion of the scheme-theoretic blow-up $\Bl_{\hat{J}}(\Spec(\hat{A}))\rightarrow \Spec(\hat{A})$. Since $A \rightarrow \hat{A}$ is flat, and since scheme-theoretic blow-ups commute with flat base change, we see that $\alpha$ is in turn the $\pi$-adic completion of the scheme-theoretic blow-up $\Bl_J(\Spec(A)) \rightarrow \Spec(A)$. 

By \autoref{lemma:AAimprovement}, there is an integral, regular  $R$-scheme $X$ flat over $R$ and a proper morphism $\rho : X \rightarrow \Bl_J(\Spec(A))$, such that the composite map $\beta : X \rightarrow \Bl_J(\Spec(A)) \rightarrow \Spec(A)$ is an isomorphism over $\Spec(A) \setminus V(J)$ and furthermore with the property that the curve $(X_{\kbar})_{\red}$ is a tree of rational curves with two irreducible components isomorphic to $\A^1_{\kbar}$ and all others isomorphic to $\P^1_{\kbar}$. The $\pi$-adic completion $\hat{\beta} : \hat{X} \rightarrow \Bl_{\hat{J}}\mathfrak{A}_0 \rightarrow \mathfrak{A}_0$ of $\beta : X \rightarrow \Bl_J(\Spec(A)) \rightarrow\Spec(A)$, is a map of admissible formal $R$-schemes, that induces an isomorphism $\hat{\beta}^{\rig} : \hat{X}^{\rig} \rightarrow \mathfrak{A}_0^{\rig} = \Ann$ of rigid-generic fibers. 
\end{proof}

\bibliographystyle{alpha}
\bibliography{main}

\newcommand{\etalchar}[1]{$^{#1}$}
\begin{thebibliography}{PST{\etalchar{+}}21}

\bibitem[And03]{Andre}
Yves Andr\'{e}.
\newblock {\em Period mappings and differential equations. {F}rom {$\Bbb C$} to {$\Bbb C_p$}}, volume~12 of {\em MSJ Memoirs}.
\newblock Mathematical Society of Japan, Tokyo, 2003.
\newblock T\^{o}hoku-Hokkaid\^{o} lectures in arithmetic geometry, With appendices by F. Kato and N. Tsuzuki.

\bibitem[Ber99]{berkovich-local-contractibility}
Vladimir~G. Berkovich.
\newblock Smooth {$p$}-adic analytic spaces are locally contractible.
\newblock {\em Invent. Math.}, 137(1):1--84, 1999.

\bibitem[Bor72]{borel}
Armand Borel.
\newblock {Some metric properties of arithmetic quotients of symmetric spaces and an extension theorem}.
\newblock {\em Journal of Differential Geometry}, 6(4):543 -- 560, 1972.

\bibitem[Bos14]{bosch}
Siegfried Bosch.
\newblock {\em Lectures on formal and rigid geometry}, volume 2105 of {\em Lecture Notes in Mathematics}.
\newblock Springer, Cham, 2014.

\bibitem[Bru22]{Brunebarbe}
Yohan Brunebarbe.
\newblock Hyperbolicity in presence of a large local system.
\newblock arXiv:v, 2022.

\bibitem[Che02]{Cherry}
William Cherry.
\newblock Non-archimedean big picard theorems, 2002.

\bibitem[Con99]{conrad-irreducible-comp}
Brian Conrad.
\newblock Irreducible components of rigid spaces.
\newblock {\em Ann. Inst. Fourier (Grenoble)}, 49(2):473--541, 1999.

\bibitem[Con18]{conradabeloids}
Brian Conrad.
\newblock Higher-level canonical subgroups in abelian varieties.
\newblock \url{https://math.stanford.edu/~conrad/papers/subgppaper.pdf}, 2018.

\bibitem[CR04]{cherry-ru}
William Cherry and Min Ru.
\newblock Rigid analytic {P}icard theorems.
\newblock {\em Amer. J. Math.}, 126(4):873--889, 2004.

\bibitem[CS19]{CaraianiScholze}
Ana Caraiani and Peter Scholze.
\newblock On the generic part of the cohomology of non-compact unitary shimura varieties.
\newblock arXiv:1909.01898, 2019.

\bibitem[Del71a]{Delignefixedpart}
Pierre Deligne.
\newblock Th\'{e}orie de {H}odge. {II}.
\newblock {\em Inst. Hautes \'{E}tudes Sci. Publ. Math.}, (40):5--57, 1971.

\bibitem[Del71b]{DeligneTravaux}
Pierre Deligne.
\newblock Travaux de {S}himura.
\newblock In {\em S\'{e}minaire {B}ourbaki, 23\`eme ann\'{e}e (1970/1971)}, volume Vol. 244 of {\em Lecture Notes in Math.}, pages Exp. No. 389, pp. 123--165. Springer, Berlin-New York, 1971.

\bibitem[dJ95]{deJong-etale}
A.~J. de~Jong.
\newblock \'{E}tale fundamental groups of non-{A}rchimedean analytic spaces.
\newblock {\em Compositio Math.}, 97(1-2):89--118, 1995.
\newblock Special issue in honour of Frans Oort.

\bibitem[DLLZ18]{DLLZ}
Hansheng Diao, Kai-Wen Lan, Ruochuan Liu, and Xinwen Zhu.
\newblock Logarithmic riemann-hilbert correspondences for rigid varieties.
\newblock arXiv:1803.05786, 2018.

\bibitem[DLLZ19]{DLLZ2}
Hansheng Diao, Kai-Wen Lan, Ruochuan Liu, and Xinwen Zhu.
\newblock Logarithmic adic spaces: some foundational results.
\newblock arXiv:1912.09836, 2019.

\bibitem[FC90]{FaltingsChai}
Gerd Faltings and Ching-Li Chai.
\newblock {\em Degeneration of abelian varieties}, volume~22 of {\em Ergebnisse der Mathematik und ihrer Grenzgebiete (3) [Results in Mathematics and Related Areas (3)]}.
\newblock Springer-Verlag, Berlin, 1990.
\newblock With an appendix by David Mumford.

\bibitem[FvdP04]{fresnel-vanderput}
Jean Fresnel and Marius van~der Put.
\newblock {\em Rigid analytic geometry and its applications}, volume 218 of {\em Progress in Mathematics}.
\newblock Birkh\"{a}user Boston, Inc., Boston, MA, 2004.

\bibitem[Gro66]{Grothabelian}
A.~Grothendieck.
\newblock Un th\'{e}or\`eme sur les homomorphismes de sch\'{e}mas ab\'{e}liens.
\newblock {\em Invent. Math.}, 2:59--78, 1966.

\bibitem[Gro72]{Grothendieck}
Alexander Grothendieck.
\newblock {\em Modeles de Neron et monodromie, in Groupes de mon- odromie en geometric algebriqu}.
\newblock Springer-Verlag, Berlin, 1972.
\newblock S{\'e}minaire de G{\'e}om{\'e}trie Alg{\'e}brique (SGA7 I), Lecture Notes in Mathematics, 288.

\bibitem[Har13]{hartl}
Urs Hartl.
\newblock On a conjecture of {R}apoport and {Z}ink.
\newblock {\em Invent. Math.}, 193(3):627--696, 2013.

\bibitem[Hir64]{hironaka-resolutionI}
Heisuke Hironaka.
\newblock Resolution of singularities of an algebraic variety over a field of characteristic zero. {I}, {II}.
\newblock {\em Ann. of Math. (2)}, 79:109--203; {\bf 79 (1964), 205--326}, 1964.

\bibitem[JV21]{javanpeykar-brodyhyperbolic}
Ariyan Javanpeykar and Alberto Vezzani.
\newblock Non-archimedean hyperbolicity and applications.
\newblock {\em J. Reine Angew. Math.}, 778:1--29, 2021.

\bibitem[Kie67]{kiehl-derham}
Reinhardt Kiehl.
\newblock Die de {R}ham {K}ohomologie algebraischer {M}annigfaltigkeiten \"{u}ber einem bewerteten {K}\"{o}rper.
\newblock {\em Inst. Hautes \'{E}tudes Sci. Publ. Math.}, (33):5--20, 1967.

\bibitem[Kis10]{Kisinintegralmodels}
Mark Kisin.
\newblock Integral models for {S}himura varieties of abelian type.
\newblock {\em J. Amer. Math. Soc.}, 23(4):967--1012, 2010.

\bibitem[L\"93]{lutkebohmert-riemann}
W.~L\"{u}tkebohmert.
\newblock Riemann's existence problem for a {$p$}-adic field.
\newblock {\em Invent. Math.}, 111(2):309--330, 1993.

\bibitem[L\"22]{lutkebohmert}
Werner L\"{u}tkebohmert.
\newblock On extension of rigid analytic objects.
\newblock {\em M\"{u}nster J. Math.}, 15(1):83--166, 2022.

\bibitem[Lip93]{lipshitz-robinson-rigid-subanalytic}
L.~Lipshitz.
\newblock Rigid subanalytic sets.
\newblock {\em Amer. J. Math.}, 115(1):77--108, 1993.

\bibitem[LS18]{KaiWen}
Kai-Wen Lan and Beno\^{i}t Stroh.
\newblock Compactifications of subschemes of integral models of {S}himura varieties.
\newblock {\em Forum Math. Sigma}, 6:Paper No. e18, 105, 2018.

\bibitem[LZ17]{liu-zhu-rigidity}
Ruochuan Liu and Xinwen Zhu.
\newblock Rigidity and a {R}iemann--{H}ilbert correspondence for $p$-adic local systems.
\newblock {\em Inventiones mathematicae}, 207(1):291--343, 2017.

\bibitem[Mor21]{Jacksonhyperbolicity}
Jackson~S. Morrow.
\newblock Non-archimedean entire curves in closed subvarieties of semi-abelian varieties.
\newblock {\em Math. Ann.}, 379(3-4):1003--1010, 2021.

\bibitem[Nad89]{Nadel}
Alan~Michael Nadel.
\newblock The nonexistence of certain level structures on abelian varieties over complex function fields.
\newblock {\em Ann. of Math. (2)}, 129(1):161--178, 1989.

\bibitem[Osw21]{p-adic-definable-chow}
Abhishek Oswal.
\newblock A non-archimedean definable chow theorem.
\newblock \textit{J. Eur. Math. Soc.}, accepted, 2021.

\bibitem[PST{\etalchar{+}}21]{PST}
Jonathan Pila, Ananth~N. Shankar, Jacob Tsimerman, an\phantom{}d an appendix by~H\'el\`ene Esnault, and Michael Groechenig.
\newblock Canonical heights on shimura varieties and the andr\'e-oort conjecture.
\newblock arXiv:2109.08788, 2021.

\bibitem[RZ96]{RZ}
M.~Rapoport and Th. Zink.
\newblock {\em Period spaces for {$p$}-divisible groups}, volume 141 of {\em Annals of Mathematics Studies}.
\newblock Princeton University Press, Princeton, NJ, 1996.

\bibitem[Shi22]{koji}
Koji Shimizu.
\newblock A {$p$}-adic monodromy theorem for de {R}ham local systems.
\newblock {\em Compos. Math.}, 158(12):2157--2205, 2022.

\bibitem[{Sta}18]{stacks-project}
The {Stacks Project Authors}.
\newblock \textit{Stacks Project}.
\newblock \url{https://stacks.math.columbia.edu}, 2018.

\bibitem[Sun20]{sun-hyperbolicity-of-mg}
Ruiran Sun.
\newblock Non-archimedean hyperbolicity of the moduli space of curves.
\newblock \url{https://arxiv.org/abs/2009.13096}, 2020.

\end{thebibliography}
\end{document}